\documentclass[12pt]{amsart}
\usepackage{amssymb,amsmath,amsthm,amsfonts,latexsym,euscript,wasysym}
\usepackage[margin = 1in]{geometry}
\newtheorem{theorem}{Theorem}
\newtheorem{definition}[theorem]{Definition}
\newtheorem{lemma}[theorem]{Lemma}
\newtheorem*{lemma*}{Lemma}
\newtheorem{corollary}[theorem]{Corollary}
\newtheorem*{corollary*}{Corollary}

\newtheorem*{theorem*}{Theorem}
\newtheorem{proposition}[theorem]{Proposition}
\newcommand{\legendre}[2]{\genfrac{(}{)}{}{}{#1}{#2}}

\usepackage[shortlabels]{enumitem}
\usepackage{hyperref}

\hypersetup{
	pdftitle = {An Elliptic Curve Analogue to the Fermat Numbers},
	bookmarks = true,
}

\newcommand{\Q}{\mathbb{Q}}
\newcommand{\Z}{\mathbb{Z}}
\newcommand{\N}{\mathbb{N}}
\newcommand{\F}{\mathbb{F}}
\newcommand{\floor}[1]{\lfloor {#1} \rfloor}
\newcommand{\denom}{\text{denom}}
\newcommand{\num}{\text{num}}
\newcommand{\ord}{\text{ord}}

\thanks{All authors were supported by NSF grant DMS-1461189.}

\title{An Elliptic Curve Analogue to the Fermat Numbers}

\author{Skye Binegar}
\author{Randy Dominick}
\author{Meagan Kenney}
\author{Jeremy Rouse}
\author{Alex Walsh}

\address{Department of Mathematics, Reed College, Portland, OR 97202}
\email{binegars@reed.edu}
\address{Department of Mathematics \& Statistics, Texas Tech University, Lubbock, TX 79409}
\email{randydominick1093@gmail.com}
\address{Department of Mathematics, Bard College, Annandale-on-Hudson, NY 12504}
\email{mk6673@bard.edu}
\address{Department of Mathematics and Statistics, Wake Forest University, Winston-Salem, NC 27109}
\email{rouseja@wfu.edu}
\address{Mathematics Department, Brown University, Providence, RI 02912}
\email{alexandra\_walsh@brown.edu}

\begin{document}

\subjclass[2010]{Primary 11G05; Secondary 11B37, 11G15, 11Y11}
\begin{abstract}
The Fermat numbers have many notable properties, including order universality, coprimality, and definition by a recurrence relation. We use arbitrary elliptic curves and rational points of infinite order to generate sequences that are analogous to the Fermat numbers. We demonstrate that these sequences have many of the same properties as the Fermat numbers, and we discuss results about the prime factors of sequences generated by specific curves and points.
\end{abstract}

\maketitle



\section{Introduction}
\label{intro sec}

In August 1640, Fermat wrote a letter to Frenicle \cite[p. 205]{FermatBook} recounting his discovery that if $n$ is not a power of $2$, then $2^{n} + 1$ is composite. Fermat also states that if $n$ is a power of $2$, then $2^{n} + 1$ is prime. As
examples, he lists the first seven numbers in this sequence, $F_{n} = 2^{2^{n}} + 1, n \geq 0$, now called the sequence of Fermat numbers.


In 1732, Euler discovered that Fermat's observation was incorrect, and that $641$ divides $F_{5} = 4294967297$. Indeed, it is now known that $F_{n}$ is composite for $5 \leq n \leq 32$. Very little is known about whether any $F_n$ are prime; heuristics suggest that only finitely many of them are prime. However, mathematicians have been unable to prove that there are infintely many composite Fermat numbers.

The primality of the Fermat numbers is connected with the classical problem of constructing a regular polygon with $n$ sides using only an unmarked straightedge and a compass. In 1801, Gauss proved that if a positive integer $n$ is a power of two multiplied by a product of distinct Fermat primes, then a regular $n$-gon is constructible with a ruler and compass. The converse of this result was proven by Wantzel in 1837. (For a modern proof, see \cite[p. 602]{DummitFoote}.)

Elliptic curves are central objects in modern number theory and have led to novel methods of factoring (see \cite{Lenstra2}), proving that numbers are prime (see \cite{AtkinMorain}), and cryptography (see
\cite{Koblitz} and \cite{Miller}). They have also
played a role in a number of important theoretical developments, including the solution of Fermat's Last Theorem (see \cite{Wiles}) and the determinantion of all integer solutions to $x^{2} + y^{3} = z^{7}$ with $\gcd(x,y,z) = 1$ (see \cite{PSS}). The present paper relies on both elliptic curves and the sequence of Fermat numbers.

We begin with our central definition:

\begin{definition}
\label{elliptic fermat definition}
For an elliptic curve $E$ and a point $P \in E(\Q)$ of infinite order, let $2^kP = \left( \dfrac{m_k}{e_k^2}, \dfrac{n_k}{e_k^3} \right)$ denote $P$ added to itself $2^k$ times under the group law on $E(\Q)$. We define the sequence of \emph{elliptic Fermat numbers} $\{F_k(E,P)\}$ as follows:

\begin{equation*}
\label{elliptic fermat def equation}
F_k(E,P) =
\begin{cases}
\dfrac{e_k}{e_{k-1}} & $if $k \geq 1 \\
e_0 & $if $k = 0.
\end{cases}
\end{equation*}

\end{definition}

One helpful aspect of this definition is that it gives us a clear relationship between $F_k(E,P)$ and $e_k$, the factor in the denominator of $2^kP$. We will make frequent use of this connection, which we state in the following lemma:

\begin{lemma}
\label{ek lemma}
For all $k \geq 0$, $e_k = F_0(E,P) \cdot F_1(E,P) \cdots F_k(E,P)$.
\end{lemma}
\begin{proof}
$F_0(E,P) \cdot F_1(E,P) \cdots F_k(E,P) = e_0 \cdot \dfrac{e_1}{e_0} \cdots \dfrac{e_k}{e_{k-1}} = e_k.$
\end{proof}

Our goal is to show that the sequence $\{F_k(E,P)\}$ strongly resembles the classic Fermat sequence. We do so by adapting properties of the classic Fermat numbers and proving that they hold for the elliptic Fermat numbers. It is well-known, for example, that any two distinct classic Fermat numbers are relatively prime, as Goldbach proved in a 1730 letter to Euler. The elliptic Fermat numbers have a similar property:

\begin{theorem}
\label{coprimality thm}
For all $k \neq \ell$, $\gcd(F_k (E, P),F_\ell(E, P))\in \lbrace 1,2 \rbrace$.
\end{theorem}

It is worth noting that for certain curves and points, we will always have $\gcd(F_k (E, P),F_\ell(E, P)) = 1$, while for all other curves and points, we will have $\gcd(F_k (E, P),F_\ell(E, P)) = 2$ for sufficiently large $k$ and $\ell$.

The classic Fermat numbers also have the useful property that for any nonnegative integer $N$, $2$ has order $2^{k+1}$ in $(\mathbb{Z}/N\mathbb{Z})^\times$ if and only if $N \mid F_0 \cdots F_k$ and $N \nmid F_0 \cdots F_{k-1}$. This property, which we call \emph{order universality}, provides a powerful connection between order and divisibility. A close parallel applies to the elliptic Fermat numbers:

\begin{theorem}
\label{elliptic fermat order universality thm}
Let $\Delta(E)$ be the discriminant of $E$ and suppose that $N$ is a positive
integer with $\gcd(N,6 \Delta(E)) = 1$. Then $P$ has order $2^k$ in $E(\mathbb{Z}/N\mathbb{Z})$ if and only if $N \mid F_0(E,P) \cdots F_k(E,P) $ and $ N \nmid F_0(E,P) \cdots F_{k-1}(E,P).$ 
\end{theorem}

In the case where $N = p$ for some odd prime $p$, we can make this statement stronger. For the classic Fermat numbers, we know that $2$ has order $2^{k+1}$ in $\F_p^\times$ if and only if $p \mid F_k$. The elliptic Fermat numbers yield the following result:

\begin{corollary}
\label{elliptic fermat order universality prime corollary}
For any odd prime $p \nmid 6\Delta(E)$, $P$ has order $2^k$ in $E(\F_p)$ if and only if $p \mid F_k(E,P).$
\end{corollary}

This corollary plays a role in several important results in the paper.

Additionally, and quite interestingly, the classic Fermat numbers can be defined by several different recurrence relations. In Section \ref{recurrence sec}, we present the following analogous result:

\begin{theorem}
\label{recurrence thm}
Let $E: y^2 = x^3 + ax^2 + bx + c$ be an elliptic curve, and let $P \in E(\mathbb{Q})$ be a point of infinite order. We can define a recurrence relation for $F_k$ by the following system of equations:
\begin{footnotesize}
\begin{align}
\label{Fk rec}
F_k(E,P) & =  \dfrac{2n_{k-1}}{\tau_k} \\
\label{nk rec}
n_k(E,P) & =  \dfrac{-2am_{k-1}m_ke_{k-1}^2 - bm_{k-1}e_{k-1}^4F_k^2 - bm_ke_{k-1}^4 - 2ce_{k-1}^6F_k^2 + m_{k-1}^3F_k^2 - 3m_{k-1}^2m_k}{\tau_k} \\
\label{mk rec}
m_k(E,P) & =  \dfrac{m_{k-1}^4 - 2bm_{k-1}^2e_{k-1}^4 - 8cm_{k-1}e_{k-1}^6 + b^2e_{k-1}^8 - 4ace_{k-1}^8}{\tau_k^2} \\
\label{ek rec}
e_k(E,P) & =  F_0 \cdot F_1 \cdot F_2 \cdots F_{k-1} \cdot F_k
\end{align}
\end{footnotesize}
\end{theorem}

Unlike the various classic Fermat recurrence relations, which only depend on previous terms, the elliptic Fermat recurrence relation we have discovered relies on several other sequences, namely $m_k$, $n_k$, $e_k$, and $\tau_k$. While the first three sequences are related to the coordinates of $2^kP$, $\tau_k$ is defined as follows:

\begin{theorem}
\label{tau def}
Let $\tau_k=\dfrac{2n_{k-1}}{F_k(E,P)}$. Then $\tau_k\in \mathbb{Z}$.
\end{theorem}

This equation follows naturally from the definition of $F_k(E, P)$ and the duplication formula, which we will see in Section \ref{background sec}. In order to have a true recurrence relation, however, we need a way to explicitly calculate $|\tau_k|$. Luckily, we know the following fact:

\begin{theorem}
\label{tau alg thm}
The $|\tau_k|$ are eventually periodic, and there is an algorithm to compute $|\tau_k|$ for all $k$.
\end{theorem}

In Section \ref{primality sec}, we address one of the most famous aspects of the classic Fermat numbers: the question of their primality. Whereas the primality of the Fermat numbers remains an open question, we have determined conditions under which we can show that there are finitely many prime elliptic Fermat numbers. We have the following theorem, where "the egg" refers to the non-identity component of the real points of the elliptic curve:

\begin{theorem}
\label{primality ek odd}
For an elliptic curve $E: y^2 = x^3 + ax^2 + bx + c$, assume the following:
\begin{enumerate}[(i)]
\item $E(\mathbb{Q}) = \langle P,T \rangle$, where $P$ has infinite order and $T$ is a rational point of order $2$.
\item $E$ has an egg.
\item $T$ is on the egg.
\item $T$ is the only integral point on the egg.
\item $P$ is not integral.
\item $\gcd(b,m_0) = 1$.
\item $|\tau_k| = 2$ for all $k$.
\item $2 \nmid e_k$ for all $k$.
\item The equations $x^4 + ax^2y^2 + by^4 = \pm 1$ has no integer solutions where $y \not \in \{0,\pm 1 \}$.
\end{enumerate}
Then $F_k(E,P)$ is composite for all $k \geq 1$.
\end{theorem}

There are choices of $E$ for which all nine of the above conditions are satisfied. For example, we can take $E:y^2=x^3-199x^2-x$. Note that $\Delta(E)$ is positive and thus $E$ has an egg \cite[p. 420]{Silverman2}. The only integral point on the curve is $T=(0,0)$, which must be on the egg because $0$ is in between the $x$-coordinates of the other two roots of the polynomial. Also, $2T=(0:1:0)$ and thus $T$ is a rational point of order $2$ on $E$. The generating point of the curve is $P = (2809/9,89623/27)$, and $\gcd(-1,2809) = 1$. Using the algorithm to compute $\tau_k$, it can be seen that $|\tau_k| = 2$ for all $k$. The Tamagawa number at $2$ is $3$ and $P$ reduces to a singular point modulo $2$. It follows that $\ell P$ reduces to a non-singular point mod $2$ if and only if $3 | \ell$, and so $e_k$ is odd for all $k$. Finally, Magma \cite{Magma} can be used to solve Thue equations in order to conclude that there are no integer solutions to $x^4 - 199x^2y^2 - y^4 = \pm 1$, where $y \not \in \{0,\pm 1 \}$. Thus this example satisfies the conditions for the theorem, and so $F_k$ is composite for all $k$. 

Section \ref{growth rate sec} focuses on the growth rate of the elliptic Fermat numbers. Much like the classic Fermat numbers, the elliptic Fermat numbers grow at a doubly exponential rate, as shown by the following theorem:

\begin{theorem}
\label{growth rate thm}
Let $F_k$ be the $k$th elliptic Fermat number in the sequence generated by the elliptic curve $E$ and the point $P = \left( \dfrac{m_0}{e_0^2}, \dfrac{n_0}{e_0^3} \right)$. If $\hat{h}(P)$ denotes the canonical height of $P$, then
$\lim\limits_{k \to \infty} \dfrac{\log(F_k)}{4^k} = \dfrac{3}{8} \hat{h}(P)$.
\end{theorem}

Finally, in Section \ref{specialcurve}, we examine the curve $E: y^2 = x^3 - 2x$ and the elliptic Fermat sequence generated by the point $P = (2,2)$. It is a theorem of Lucas that a prime divisor of the Fermat sequence is congruent to $1 \pmod{2^{n+2}}$. Upon examination of the factorization of the numbers in the sequence $\{F_n(E,P)\}$, we arrive at a pleasing congruence analogue.

\begin{theorem}\label{cong}
Let $E: y^2 = x^3 - 2x$ and consider the point $P = (2,2)$ and the elliptic Fermat sequence $(F_n (E,P))$. For any prime $p$ such that $p \vert F_k(E,P)$ for some $k$, we have
$$p \equiv \begin{cases} 
1 \pmod{2^n} & \text{ if } p \equiv 1 \pmod 4\\
-1 \pmod{2^n} & \text{ if } p \equiv -1 \pmod 4.
\end{cases}
$$
\end{theorem}

In addition to this congruence result, we have a partial converse that tells us about the presence of Fermat and Mersenne primes in $(F_n(E,P))$:

\begin{theorem}\label{congversef}
For $E: y^2 = x^3 - 2x$, consider the point $P = (2,2)$. Let $F_k = 2^{2^k}+1$ be a Fermat prime and $F_k \neq 5,17$. Then $F_k$ divides $F_n(E,P)$ for some $n \leq 2^{k-1} - 1$. 
\end{theorem}

\begin{theorem}\label{congversem}
For $E: y^2 = x^3 - 2x$, consider the point $P = (2,2)$. Let $q = 2^p - 1 \geq 31$ be a Mersenne prime. Then $q$ divides $F_n(E,P)$ for some $n \leq p-3 \in \N$.
\end{theorem}

\subsection*{Acknowledgements}
We would like to thank the Wake Forest Department of Mathematics and Statistics for their hospitality and resources. We would also like to thank Magma version 2.22-9 \cite{Magma} and Sage version 7.5.1 \cite{sagemath}, which we used for computations.


\section{Background}
\label{background sec}

We begin with some general background on elliptic curves. For the purposes of this paper, an elliptic curve is a non-singular cubic curve defined over $\Q$ that has the form $y^2 = x^3 + ax^2 + bx + c$ for some $a, b, c \in \Z$. When we say $E$ is non-singular, we mean that there are no singular points on the curve. We will often think of $E$ as living in $\mathbb{P}^{2}$ and represent it with
the homogeneous equation $y^{2} z = x^{3} + ax^{2}z + bxz^{2} + cz^{3}$.

A \emph{singular point} is a point $P=(x:y:z)$ at which there is not a well-defined tangent line. These points occur when the following equations are equal to 0: 
\begin{eqnarray}
F(x,y,z) &=& y^2z-x^3-ax^2z-bx^2z-cz^3 \\ \nonumber
\frac{\partial F}{\partial x} &=& -3x^2-2azx-bz^2 \\ \nonumber
\frac{\partial F}{\partial y} &=& 2yz \\ \nonumber
\frac{\partial F}{\partial z} &=& y^2-ax^2-2bxz-3cz^2.
\end{eqnarray}

We write $E(\Q)$ to denote the set of rational points on $E$ along with the point at infinity, $(0 : 1 : 0)$. Using the following binary operation, we can give $E(\Q)$ a group structure: for $P, Q \in E(\Q)$, draw a line through $P$ and $Q$ and let $R = (x, y)$ be the third intersection point of the line with the curve. Then $P + Q = (x, -y)$. This operation gives an abelian group structure on $E(\Q)$ with $(0 : 1 : 0)$ as the identity.

Any $P \in E(\Q)$ can be expressed in projective space as $P = \left( \frac{m}{e^2} : \frac{n}{e^3} : 1 \right)$ = $(me : n : e^3)$ for some $m, n, e \in \Z$ with $\gcd(m, e) = \gcd(n, e) = 1.$ We can reduce each $P \in E(\Q)$ mod $p$ to a point in $E(\F_p)$ as follows: $P \bmod{p} = (me \bmod{p} : n \bmod{p} : e^3 \bmod{p})$. If $E/\F_p$ is non-singular, then the map from $E(\Q)$ to $E(\F_p)$ given by $P \mapsto P \bmod{p}$ is a homomorphism. We remark that $P \bmod{p} = (0 : 1 : 0)$ if and only if $p \mid e$.

Let $\mathbb{Q}_p$ be the field of $p$-adic numbers. The following sets are subgroups of $E(\mathbb{Q}_p)$:

\begin{eqnarray}
E_0(\mathbb{Q}_p)&=&\{P\in E(\mathbb{Q}_p)\mid P \hbox{ reduces to a non-singular point}\} \\ \nonumber
E_1(\mathbb{Q}_p)&=&\{P\in E(\mathbb{Q}_p)\mid P \hbox{ reduces to } (0:1:0)\pmod{p}\}.
\end{eqnarray}
Note that  $E_1(\mathbb{Q}_p)\subseteq E_0(\mathbb{Q}_p)\subseteq E(\mathbb{Q}_p)$.
It is also important to note that $[E(\mathbb{Q}_p):E_0(\mathbb{Q}_p)]$ is finite and is called the \emph{Tamagawa number} of $E$ at $p$.

Another important characteristic of elliptic curves is the discriminant. The \emph{discriminant} of an elliptic curve $E$ is defined as $\Delta(E) =64a^3c+16a^2b^2+288abc-64b^3-432c^2$, and it can tell us quite a bit about $E$. For example, when considering $E(\mathbb{R})$, curves can have either one or two components. We refer to the connected component of the identity as the \emph{nose}. If there is a second component, we refer to it as the \emph{egg}. The discriminant of $E$ is positive, if and only if $E$ has an egg \cite[p. 420]{Silverman2}. For a curve with two components, let $P_{egg},\text{ } Q_{egg}$ be points on the egg, and let $P_{nose},\text{ } Q_{nose}$ be points on the nose. Then $P_{egg} + Q_{egg}$ and $P_{nose} + Q_{nose}$ are on the nose, while $P_{egg} + P_{nose} = P_{nose} + P_{egg}$ is on the egg.



Since our definition of the elliptic Fermat numbers involves doubling points, it is convenient to use the notation $2^kP = \left(\frac{m_k}{e_k^2}, \frac{n_k}{e_k^3}\right)$. We also rely on the \emph{duplication formula} expressing
the $x$-coordinate of $2Q$ in terms of that of $Q$. In particular,
if $2^{k-1} P = (x_{k-1},y_{k-1})$, Silverman and Tate \cite[p. 39]{SilvermanTate} give:

\begin{equation*}
\label{duplication x y}
X(2^kP) = \dfrac{x_{k-1}^4 - 2bx_{k-1}^2 - 8cx_{k-1} + b^2 - 4ac}{4(x_{k-1}^3 + ax_{k-1}^2 + bx_{k-1} + c)}.
\end{equation*}

Letting $2^{k-1} P = \left( \dfrac{m_{k-1}}{e_{k-1}^2} , \dfrac{n_{k-1}}{e_{k-1}^3} \right)$, we can put this in terms of $m_{k-1}$ and $e_{k-1}$:

\begin{equation}
\label{duplication m e}
X(2^kP) = \dfrac{m_{k-1}^4-2bm_{k-1}^2e_{k-1}^4-8cm_{k-1}e_{k-1}^6+b^2e_{k-1}^8-4ace_{k-1}^8}{4e_{k-1}^2(m_{k-1}^3+am_{k-1}^2e_{k-1}^2+bm_{k-1}e_{k-1}^4+ce_{k-1}^6)}.
\end{equation}

Since $y^2 = x^3 + ax^2 + bx + c$, substituting $2^kP = (\frac{m_k}{e_k^2},\frac{n_k}{e_k^3})$ gives us

\begin{equation}
\label{n^2 eq}
n_k^2 = m_k^3 + am_k^2e_k^2 + bm_ke_k^4 + ce_k^6.
\end{equation}

Combining \eqref{duplication m e} and \eqref{n^2 eq}, we get our final duplication formula:

\begin{equation}
\label{duplication m n e}
X(2^kP) = \dfrac{m_{k-1}^4-2bm_{k-1}^2e_{k-1}^4-8cm_{k-1}e_{k-1}^6+b^2e_{k-1}^8-4ace_{k-1}^8}{4n_{k-1}^2e_{k-1}^2}.
\end{equation}

We will refer to the unreduced numerator and denominator in the above equation as $A$ and $B$, respectively, i.e.

\begin{gather}
\label{duplication A}
A = m_{k-1}^4-2bm_{k-1}^2e_{k-1}^4-8cm_{k-1}e_{k-1}^6+b^2e_{k-1}^8-4ace_{k-1}^8 \\
\label{duplication B}
B = 4n_{k-1}^2e_{k-1}^2.
\end{gather}

One last aspect of elliptic curves that will prove useful in section \ref{specialcurve} is the concept of complex multiplication. We say that an elliptic curve has \emph{complex multiplication} if its endomorphism ring is isomorphic to an order in an imaginary quadratic field. In other words, $E$ is equipped with more maps than simple integer multiplication of a point, and composition of these maps is similar to multiplication in an imaginary quadratic field.

Complex multiplication is relevant to our work because it allows us to count the points on the curve over finite fields. In the final section, we will study the curve $E: y^2 = x^3 - 2x$, and our results rely on having a good understanding of $|E(\F_p)|$. As a special case of Proposition 8.5.1 from Cohen \cite[p. 566]{Cohen}, we have the following fact about our curve $E$:

\begin{proposition}\label{specialorder}

Let $E: y^2 = x^3 - 2x$ be an elliptic curve and let $p$ be an odd prime. Then $|E(\F_p)| = p+1 - a_p(E)$, where $a_p(E)$ is known as the \emph{trace of Frobenius} of an elliptic curve over a prime. When $p \equiv 3 \pmod 4$, we have $a_p(E) = 0$. If $p \equiv 1 \pmod 4$, then

$$
a_p(E) = 2 \legendre{2}{p}
\begin{cases}
-a, & \text{ if } 2^{(p-1)/4} \equiv 1 \pmod p \\
a, & \text{ if } 2^{(p-1)/4} \equiv -1 \pmod p \\
-b, & \text{ if } 2^{(p-1)/4} \equiv -a/b \pmod p \\
b, & \text{ if }2^{(p-1)/4} \equiv a/b \pmod p \\
\end{cases}
$$
where $a$ and $b$ are integers such that $p = a^2 + b^2$ with $a \equiv -1 \pmod 4$.
\end{proposition}

\section{Coprimality}
\label{coprimality sec}

In this section, we will prove Theorem \ref{coprimality thm} using Lemma \ref{ek lemma} and the duplication formula \eqref{duplication m n e}. Note that in this section, $F_k$ refers to the $k$th elliptic Fermat number.

\begin{proof}
Let $E: y^2 = x^3 + ax^2 + bx + c$ be an elliptic curve, and let $2^{k-1}P = \left( \dfrac{m_{k-1}}{e_{k-1}^2}, \dfrac{n_{k-1}}{e_{k-1}^3} \right)$ with $k \geq 1$ be a point in $E(\mathbb{Q})$.
Since we know from Lemma \ref{ek lemma} that $e_{k-1} = F_0 \cdot F_1 \cdot F_2 \cdots F_{k-2} \cdot F_{k-1}$, showing that $\gcd(F_k,e_{k-1})\in \lbrace 1,2 \rbrace$ is sufficient to prove Theorem \ref{coprimality thm}.
Recall the duplication formula:
\begin{equation*}
X(2^kP) = \dfrac{m_{k-1}^4-2bm_{k-1}^2e_{k-1}^4-8cm_{k-1}e_{k-1}^6+b^2e_{k-1}^8-4ace_{k-1}^8}{4n_{k-1}^2e_{k-1}^2}.
\end{equation*}
So $e_k^2 \mid 4n_{k-1}^2e_{k-1}^2$ and thus $F_k \mid 2n_{k-1}$.
Now, since $\gcd(n_{k-1},e_{k-1}) = 1$, if $2 \mid F_k$ and $2 \mid e_{k-1}$, then $\gcd(F_k,e_{k-1}) = 2$. Otherwise we must have $\gcd(F_k,e_{k-1}) = 1$.
\end{proof}

Note that if $2 \mid e_t$ for some $t$, then $2 \mid e_k$ for all $k \geq t$. Not only that, but the power of $2$ that divides $e_k$ will increase as $k$ increases. Thus $2 \mid F_k$ for all $k \geq t$. So in this case, $\gcd(F_k,F_\ell) = 2$ for all $k$, $\ell \geq t$, $k \neq \ell$. Otherwise, if $2 \nmid e_k$ for all $k$, then $\gcd(F_k,F_\ell) = 1$ for all $k \neq \ell$.

\section{Order Universality}
\label{order universality sec}

The proof of Theorem \ref{elliptic fermat order universality thm}, which is itself fairly straightforward, requires the existence of a homomorphism from $E(\Q)$ to $E(\Z/N\Z)$. We know that this homomorphism exists when $N = p$ for an odd prime $p$ with $p \nmid \Delta(E)$, since in that case we are working with $E(\F_p)$. When $N$ is not prime, however, we need to define a group structure on elliptic curves over finite rings before we can talk about such a map. To do so, we adapt the group structure of elliptic curves over fields, as discussed by Lenstra \cite{Lenstra}. We can define a group structure on $E(\Z/N\Z)$ provided that the following conditions hold:

\begin{enumerate}
\item $\gcd(N, 6\Delta(E)) = 1$;
\item For any primitive $m \times n$ matrix with entries in $\Z/N\Z$ whose $2 \times 2$ subdeterminants are all zero, there exists a linear combination of the rows that is primitive in $\Z/N\Z$.
\end{enumerate}

\indent We say that a finite collection of elements $(a_i)$ of a ring $R$ is $primitive$ if it generates $R$ as an $R$-ideal, that is, if there exist $b_i \in R$ such that $\Sigma b_ia_i = 1$. A matrix is primitive if its entries are primitive in $R$. We note that condition (2) holds for any finite ring and is therefore true no matter which $N$ we choose. 

Assume $N$ satisfies the above conditions, and let $S$ and $T$ be points in $E(\Z/N\Z)$ given by $S = (x_1 : y_1 : z_1)$ and $T = (x_2 : y_2 : z_2).$ Suppose $S \neq (0 : 1 : 0)$ or $T \neq (0 : 1 : 0)$. If $N = p$ for some odd prime $p$, then $\Z/N\Z$ is the field $\F_p$, and we can define the line connecting $S$ and $T$ in the standard way, i.e. by using one of two linear equations, the choice of which depends on if $x_1 = x_2$ or $y_1 = -y_2$. Each equation will give a formula for $S + T$, respectively denoted by $(q_1 : r_1 : s_1)$ and $(q_2 : r_2 : s_2)$, where $q_i, r_i, s_i$ are polynomial expressions in terms of $x_i, y_i$ and $z_i$. Neither formula is defined in the case where $S = T = (0 : 1: 0)$, but it is simple enough to let $S + T = (0 : 1 : 0)$.

If $N$ is not prime, on the other hand, then two equations do not suffice. Whereas $S = T = (0 : 1 : 0)$ over a field $\mathbb{F}_p$ only when $S \equiv T \equiv (0 : 1 : 0) \pmod{p}$, there is more potential for trouble over a ring. Suppose, for example, that $N = pq$ for distinct primes $p$ and $q$. It is then possible that $S \equiv T \equiv (0 : 1 : 0) \pmod{p}$ but $S \not\equiv (0 : 1 : 0) \pmod{q}$ or $T \not\equiv (0 : 1 : 0) \pmod{q}$. In this event, neither formula applies mod $p$, but it is not the case that $S \equiv T \equiv (0 : 1 : 0) \pmod{N}$. To account for these dangerous possibilities, we include a third equation, which in turn yields a new formula for $S + T$, denoted by $(q_3 : r_3 : s_3)$.  We then have nine polynomial expressions, $q_i, r_i, s_i$ for $i \in \{1,2,3\}$, the explicit formulae for which are stated by Lange and Ruppert \cite{LangeRuppert}.

With these polynomials in hand, we now consider the $3 \times 3$ matrix given by 

$$A = \begin{pmatrix}

q_1 & r_1 & s_1 \\
q_2 & r_2 & s_2 \\
q_3 & r_3 & s_3

\end{pmatrix}.$$

\vspace{0.1in}

The matrix $A$ is primitive, and all of its $2 \times 2$ subdeterminants are zero \cite{Lenstra}. Thus, by condition (2) above, there exists a linear combination of rows, $(q_0, r_0, s_0)$, that is primitive in $(\Z/N\Z)$. This linear combination is uniquely determined up to multiplication by units. We can thus define the sum of $S$ and $T$ to be $(q_0 : r_0 : s_0)$. As Lenstra notes \cite{Lenstra}, the other group axioms follow from the definition of this operation. Hence we have defined a group structure on $E(\Z/N\Z)$.

Applying this group structure to $E(\Q)$ allows us to define a homomorphism from $E(\Q)$ to $E(\Z/N\Z)$, just as we desired.

\begin{lemma}
\label{modNhom}
The map $\phi : E(\Q) \rightarrow E(\Z/N\Z)$ given by $P \mapsto P \bmod{N}$ is a homomorphism.

\end{lemma}

\begin{proof}

Let $P$ and $Q$ be points in $E(\Q)$ given by $P = (x_1 : y_1 : z_ 1)$ and $Q = (x_2 : y_2 : z_ 2)$. Scale $P$ and $Q$ so that $x_i, y_i, z_i$ are integers with $\gcd(x_1, y_1, z_1) = \gcd(x_2, y_2, z_2) = 1$. Now $(x_1, y_1, z_1)$ and $(x_2, y_2, z_2)$ are primitive in $\Z$, so we are essentially working in $E(\Z) \subseteq E(\Q)$. We construct a $3 \times 3$ matrix $A$ whose entries are the polynomial expressions described earlier, which we denote by $q_i, r_i, s_i$ for $i \in \{1,2,3\}$. Since $A$ is primitive with all $2 \times 2$ subdeterminants equal to $0$, we can find a primitive $\Z$-linear combination $(q_0, r_0, s_0)$ of its rows that will yield the point $P + Q = (q_0 : r_0 : s_0)$. Thus $\phi(P + Q) \equiv (q_0 : r_0 : s_0) \pmod{N}$.

Next, we calculate $\phi(P) + \phi(Q)$. Since $\phi(P) \equiv P \bmod{N}$ and $\phi(Q) \equiv Q \bmod{N}$, and since $f(x) \equiv f(x\pmod{N}) \pmod{N}$ for any polynomial $f(x)$, the values of the nine polynomials we seek will be the same as those defined above, mod $N$. So the entries of the resulting matrix $A'$ will be exactly the entries of $A$, mod $N$:

$$A' = \begin{pmatrix}

q_1 \bmod{N} & r_1 \bmod{N} & s_1 \bmod{N} \\
q_2 \bmod{N} & r_2 \bmod{N} & s_2 \bmod{N} \\
q_3 \bmod{N} & r_3 \bmod{N} & s_3 \bmod{N}

\end{pmatrix}.$$

Since $(q_0, r_0, s_0)$ is primitive in $\Z$, we know there exist some $k_1, k_2, k_3 \in \Z$ such that $k_1(q_0q_1 + r_0q_2 + s_0q_3) + k_2(q_0r_1 + r_0r_2 + s_0r_3) + k_3(q_0s_1 + r_0s_2 + s_0s_3) = 1$. This implies that $k_1(q_0q_1 + r_0q_2 + s_0q_3) + k_2(q_0r_1 + r_0r_2 + s_0r_3) + k_3(q_0s_1 + r_0s_2 + s_0s_3) \equiv 1 \pmod{N} = 1$. Thus $(q_0 \bmod{N}, r_0 \bmod{N}, s_0 \bmod{N})$ is a primitive linear combination of the rows of $A'$. It follows that $\phi(P) + \phi(Q) = (q_0 \bmod{N} : r_0 \bmod{N} : s_0 \bmod{N}) = \phi(P + Q)$, and we have shown that $\phi$ is a homomorphism.

\end{proof}

With this homomorphism in place, we are finally in a position to directly approach the proof of Theorem \ref{elliptic fermat order universality thm}.

\begin{proof}

Let $P \in E(\Q)$ be a point of infinite order and $k$ a nonnegative integer. Recall that we denote $2^kP = (m_ke_k : n_k : e_k^3)$ for $m_k, n_k, e_k \in \Z$ with $\gcd(m_k,e_k) = \gcd(n_k, e_k) = 1$. Suppose $N$ is a nonnegative integer with $\gcd(N, 6\Delta(E)) = 1$, and define $\phi$ as the homomorphism from $E(\Q)$ to $E(\Z/N\Z)$ given by $P \mapsto P \bmod{N}$.

We begin by assuming $N \mid F_0(E, P) \cdots F_k(E,P) \text{ and } N \nmid F_0(E,P) \cdots F_{k-1}(E,P)$. By Lemma \ref{ek lemma}, $N \mid e_k$. So $2^k\phi(P) = \phi(2^kP) = (m_ke_k \bmod{N} : n_k \bmod{N} : e_k^3 \bmod{N}) = (0 : 1: 0).$ It follows that the order of $\phi(P)$ divides $2^k$. If $2^{k-1}\phi(P) = (0 : 1 : 0)$, then $N \mid m_{k-1}e_{k-1}$ and $N \mid {e^3_{k-1}}$. Since $N \mid {e^3_{k-1}}$ and $\gcd(m_{k-1},e_{k-1}) = 1$, it must be the case that $\gcd(N, m_{k-1}) = 1$. Thus $N \mid m_{k-1}e_{k-1}$ implies $N \mid e_{k-1}$. But by assumption, $N \nmid F_0(E,P) \cdots F_{k-1}(E,P) = e_{k-1}$, so we have a contradiction. Moreover, since $2^{k-1}P \neq (0 : 1 : 0)$, $2^s\phi(P) \neq (0 : 1 : 0)$ for any $s$ $\textless$ $k$. Hence $\phi(P)$ has order $2^k$.

Conversely, assume $\phi(P)$ has order $2^k$. Then $2^k\phi(P) = (0 : 1 : 0)$ and $2^{k-1}\phi(P) \neq (0 : 1 : 0)$. By a similar argument as above, $2^k\phi(P) = (0 : 1 : 0)$ implies $N \mid e_k = F_0(E,P) \cdots F_k(E,P)$. Furthermore, it follows from $2^{k-1}\phi(P) \neq (0 : 1 : 0)$ that $N \nmid m_{k-1}e_{k-1}$ or $N \nmid {e^3_{k-1}}$. In either case, $N \nmid e_{k-1}.$ Thus $N \nmid F_0(E, P) \cdots F_{k-1}(E, P)$, and we are done.

\end{proof}


If $N = p$ for some odd prime $p$, then the proof of Corollary \ref{elliptic fermat order universality prime corollary} follows naturally:

\begin{proof}

Since $p$ is prime, the assumptions $p \mid F_0(E, P) \cdots F_k(E, P)$ and $p \nmid F_0(E, P) \cdots F_{k-1}(E, P)$ imply $p \mid F_k(E, P)$. Conversely, if we suppose $p \mid F_k(E, P)$, then clearly $p \mid F_0(E, P) \cdots F_k(E, P)$. In addition, Theorem \ref{coprimality thm} tells us that $p \nmid F_i(E, P)$ for all $i \neq k$. So $p \nmid F_0(E, P) \cdots F_{k-1}(E, P)$. Hence $p \mid F_k(E, P)$ if and only if $p \mid F_0(E, P) \cdots F_k(E, P)$ and $p \nmid F_0(E, P) \cdots F_{k-1}(E, P)$. Applying Theorem \ref{elliptic fermat order universality thm} completes the proof. 

\end{proof}


\section{Recurrence}
\label{recurrence sec}

In this section, we will explore the recurrence relation given by Theorem \ref{recurrence thm}. Before looking further into the recurrence relation, we must have a better understanding of the sequence $\tau_k$. Recall that $\tau_kF_k=2\num(Y(2^{k-1}P))$.

\begin{proof}[Proof of Theorem \ref{tau def}]
The duplication formula gives us that $$X(2^{k}P) = \dfrac{m_{k-1}^4-2bm_{k-1}^2e_{k-1}^4-8cm_{k-1}e_{k-1}^6+b^2e_{k-1}^8-4ace_{k-1}^8}{4n_{k-1}^2e_{k-1}^2}.$$

It cannot be assumed that $X(2^{k}P)$ is reduced in this form; however, it can be seen that $e_k=\sqrt{\denom(X(2^kP))}$ divides $2n_{k-1}e_{k-1}$. 
 
Note that $\gcd(m_{k-1},e_{k-1})=1$, which implies that $\gcd(e_{k-1},\num(X(2^kP)))=1$ and thus $e_{k-1}\vert \sqrt{\denom(X(2^kP))}$. Hence $\dfrac{\sqrt{\denom(X(2^{k}P))}}{{e_{k-1}}}$ divides $2n_{k-1}$. Observe that $2n_{k-1}=2\num(Y(2^{k-1}))$ and $F_k(E,P)=\dfrac{e_k}{e_{k-1}}$, and thus $F_k(E,P)$ divides $2\num(Y(2^{k-1}P))$. Therefore, there exists some $\tau_k\in \mathbb{Z}$ such that $F_k\tau_k=2\num(Y(2^{k-1}P))$.
\end{proof}

\begin{corollary}
\label{tau gcd}
For all $k\ge 1$, we have $\tau_{k}^2=\gcd(A,B)$, where $A$ and $B$ are defined by \eqref{duplication A} and \eqref{duplication B}.
\end{corollary}
\begin{proof}
Using the definition for $\tau_k$, we can see that
\begin{eqnarray}
\tau_k^2 &=&\dfrac{4(\num(Y(2^{k-1}P)))^2}{F_k(E,P)^2} \nonumber \\
 &=&\dfrac{4(\num(Y(2^{k-1}P)))^2\cdot e_{k-1}^{2}}{e_{k}^{2}}.\nonumber
\end{eqnarray}

Recall that $A=m_{k-1}^4-2bm_{k-1}^2e_{k-1}^4-8cm_{k-1}e_{k-1}^6+b^2e_{k-1}^8-4ace_{k-1}^8$ and  $B=4n_{k-1}^2e_{k-1}^2$. Then $\frac{m_{k}}{e_{k}^{2}} = X(2^{k} P) =\dfrac{A}{B}=\dfrac{\left(\dfrac{A}{\gcd(A,B)}\right)}{\left(\dfrac{B}{\gcd(A,B
)}\right)}$. Hence, $e_{k}^{2} = \frac{B}{\gcd(A,B)}$ and
\begin{eqnarray}
\tau_k^2 &=&\dfrac{4n_{k-1}^2e_{k-1}^2}{\left(\dfrac{4n_{k-1}^2e_{k-1}^2}{gcd(A,B)}\right)}=\gcd(A,B).\nonumber 
\end{eqnarray}
\end{proof}

We will now prove Theorem \ref{recurrence thm}. For now, keep in mind that we can explicitly calculate $\tau_k$ for all $k$; we will prove this at the end of the section. We can see that \eqref{Fk rec} comes directly from the definition of $\tau_k$ given in Section \ref{intro sec}. Now since \eqref{ek rec} was already proven as Lemma \ref{ek lemma}, we only need to show the correctness of \eqref{nk rec} and \eqref{mk rec}, which we will do in separate lemmas.

\begin{lemma}
\label{nk rec lemma}
Equation \eqref{nk rec} is correct.
\end{lemma}

\begin{proof}
From the formulas given by Silverman \cite[p. 58-59]{Silverman}, we can see that
\begin{equation*}
Y(2^kP) = \dfrac{-2am_{k-1}m_ke_{k-1}^4 - bm_{k-1}e_{k-1}^4e_k^2 - bm_ke_{k-1}^6 - 2ce_{k-1}^6e_k^2 + m_{k-1}^3e_k^2 - 3m_{k-1}^2m_ke_{k-1}^2}{2n_{k-1}e_{k-1}^3e_k^2}.
\end{equation*}
Then since $Y(2^kP) = \dfrac{n_k}{e_k^3}$,
\begin{eqnarray*}
n_k & = & Y(2^kP) \cdot e_k^3 \\
 & = & \dfrac{-2am_{k-1}m_ke_{k-1}^4e_k^3 - bm_{k-1}e_{k-1}^4e_k - bm_ke_{k-1}^6e_k - 2ce_{k-1}^6e_k^3 + m_{k-1}^3e_k^3 - 3m_{k-1}^2m_ke_{k-1}^2e_k}{2n_{k-1}e_{k-1}^3}.
\end{eqnarray*}
Then using the previously established equations $2n_{k-1} = F_k\tau_k$ and $F_k = \dfrac{e_k}{e_{k-1}}$, we can simplify this to
\begin{equation*}
n_k(E,P) =  \dfrac{-2am_{k-1}m_ke_{k-1}^2 - bm_{k-1}e_{k-1}^4F_k^2 - bm_ke_{k-1}^4 - 2ce_{k-1}^6F_k^2 + m_{k-1}^3F_k^2 - 3m_{k-1}^2m_k}{\tau_k}.
\end{equation*}
\end{proof}

\begin{lemma}
\label{mk rec lemma}
Equation \eqref{mk rec} is correct.
\end{lemma}

\begin{proof}
Recall that the duplication formula \eqref{duplication m n e} for the $x$-coordinate of the $2^kP$ is given as follows:

\begin{equation*}
X(2^kP) = \dfrac{m_{k-1}^4-2bm_{k-1}^2e_{k-1}^4-8cm_{k-1}e_{k-1}^6+b^2e_{k-1}^8-4ace_{k-1}^8}{4e_{k-1}^2n_{k-1}^2}.
\end{equation*}

Then by Corollary \ref{tau gcd}, $\tau_k$ is the gcd of the numerator and denominator in this equation, we have that
\begin{equation*}
m_k(E,P) = \dfrac{m_{k-1}^4-2bm_{k-1}^2e_{k-1}^4-8cm_{k-1}e_{k-1}^6+b^2e_{k-1}^8-4ace_{k-1}^8}{\tau_k^2}.
\end{equation*}
\end{proof}

We can now see that the recurrence relation is correct, thus proving Theorem \ref{recurrence thm}. Now we just need a better understanding of $\tau_k$ in order to show that we can calculate $\tau_k$ for all $k$.

First we will look at the relationship between the $\tau_k$ sequence and the discriminant of an elliptic curve. We can do this by looking at the discriminant of the cubic.

 Silverman and Tate \cite[p.56]{SilvermanTate} define the discriminant of the cubic as $D=-4a^3c+a^2b^2+18abc-4b^3-27c^2$. Note that the discriminant of an elliptic curve $\Delta(E)$ is $16D$.

\begin{lemma}
\label{disc of cubic lemma}
The number $\tau_k^2\vert \frac{\Delta(E)}{4}$.
\end{lemma}

\begin{proof}
Let $f(x)=x^{3}+ax^2+bx+c$, $F(x)=3x^3-ax^2-5bx+2ab-27c$, $\phi(x)=x^4-2bx^2-8cx+b^2-4ac$, and $\Phi(x)=-3x^2-2ax+a^2-4b$. Silverman and Tate \cite[p.62]{SilvermanTate} show us that $D=f(x)F(x)+\phi(x)\Phi(x)$. Plugging in our $X(2^{k-1}P)$, we observe that $$De_{k-1}^{12}=\left(e_{k-1}^6f\left(\frac{m_{k-1}}{e_{k-1}^2}\right)\right)\cdot \left(e_{k-1}^6F\left(\frac{m_{k-1}}{e_{k-1}^2}\right)\right)+\left(e_{k-1}^4\Phi\left(\frac{m_{k-1}}{e_{k-1}^2}\right)\right)\cdot \left(e_{k-1}^8\phi\left(\frac{m_{k-1}}{e_{k-1}^2}\right)\right).$$

Recall that $\tau_k^2=\gcd(A,B)$ where $A$ and $B$ are given by \eqref{duplication A} and \eqref{duplication B}. Note that  $e_{k-1}^8\phi\left(\dfrac{m_{k-1}}{e_{k-1}^2}\right)=A$ and $4e_{k-1}^2\cdot\left(e_{k-1}^6f\left(\dfrac{m_{k-1}}{e_{k-1}^2}\right)\right)=B$. Multiplying through by $4e_{k-1}^2$, we see that $\tau_{k}^2\vert 4De_{k-1}^{14}$.

Since  $\gcd(A,e_{k-1})=1$ and $e_{k-1}\vert B$, it can be seen that $\gcd(\tau_k^2,e_{k-1})=\gcd(\gcd(A,B),e_{k-1})=\gcd(A,\gcd(B,e_{k-1}))=\gcd(A,e_{k-1})=1$. It follows that $\tau_k^2\vert 4De_{k-1}^{14}$ implies that $\tau_k^2\vert 4D$ and $4D=\frac{\Delta(E)}{4}$. 

\end{proof}

In addition to being connected to the discriminant, the $\tau_k$ sequence is related to how points on the elliptic curve reduce mod a prime.

\begin{theorem}
\label{singular point 1 thm}
Suppose that $p\vert \tau_k$ and $p$ is an odd prime. Then $2^{k-1}P$ reduces to a singular point mod $p$ with $Y(2^{k-1}P)\equiv 0 \pmod{p}$.
\end{theorem}

\begin{proof}
For the remainder of this proof, $2^{k-1}P=(x,y)\pmod{p}$ will refer to the point when it has been reduced mod $p$.
Recall that $\tau_k F_k=2\num(Y(2^{k-1}P))$, and thus $p\vert 2\num(Y(2^{k-1}P))$. And $p\nmid 2$ because $p$ is an odd prime, in which case $p\vert \num(Y(2^{k-1}P))$. This tells us that $y\equiv 0 \pmod{p}.$

Let $F(x,y,z)=y^2z-x^3-ax^2z-bxz^2-cz^3$. Recall that singular points occur when $F=\frac{\partial F}{\partial x}=\frac{\partial F}{\partial y}=\frac{\partial F}{\partial z}=0$. Observe 
\begin{align*}
  \frac{\partial F}{\partial x}&= -3x^2-2azx-bz^2, \quad
\frac{\partial F}{\partial y} = 2yz\\
\frac{\partial F}{\partial z}&= y^2-ax^2-2bxz-3cz^2.\\
\end{align*}

Note that $\frac{\partial F}{\partial y}=0$ because $y\equiv 0 \pmod{p}$. Also, $z=1$ or $z=0$, but in this case $z=1$ because otherwise $\gcd(n_{k-1},e_{k-1})>1$, which would be a contradiction. 

It can be seen that $y^2=f(x)\equiv 0 \pmod{p}$. Observe that $F(x,y,z)=-f(x)$ and so $F(x,y,z)\equiv 0 \pmod{p}$.

Let $A$ and $B$ be the unreduced numerator and denominator of $X(2^kP)$ as defined in \eqref{duplication A} and \eqref{duplication B}. Since $p\vert \tau_k$ and $\tau_{k}^2\vert A$, then $p\vert A$. Note that $A\equiv f'(x)^2-(8x+4a)(f(x))\pmod{p}$ and $p\vert f(x)$. Therefore it must be the case that $p\vert f'(x)$. Since $\frac{\partial F}{\partial x}=-f'(x)$, $\frac{\partial F}{\partial x}\equiv 0 \pmod p$.

Now, setting $F(x,y,z)$ and $\frac{\partial F}{\partial x}$ equal to $0$, we can solve for $b$ and $c$. We find that $b=-3x^2-2ax$ and $c=-x^3-ax^2-bx$.  Thus substituting shows us that $\frac{\partial F}{\partial z}= -ax^2-2bx-3c=6x^3+3x^3-9x^3-ax^2+4ax^2+3ax^2-6ax^2=0.$

In summary, we know that $y\equiv 0\pmod{p}$ and thus $\frac{\partial F}{\partial y}=0$, which tells us that both $F(x,y,z)=0$ and $\frac{\partial F}{\partial x}=0$, which implies that $\frac{\partial F}{\partial z}=0$. Therefore $2^{k-1}P$ reduces to a singular point $\mod{p}$.
\end{proof}

 We can also look at a partial converse of this theorem. Although it requires an extra condition, it allows us to make conclusions about what each $\tau_k$ is based on which points on the curve reduce to singular points mod a prime.

\begin{theorem}
\label{singular point partial converse thm}
Let $p$ be an odd prime. Suppose that $2^{k-1}P$ and $2^kP$ both reduce to singular points mod $p$. Then $p\vert \tau_k$.
\end{theorem}

\begin{proof}
 We have $2^{k-1}P\equiv(x,y)\pmod{p}$ for some $x,y\in \mathbb{Z}$. If $2^{k-1}P$ is singular than we know that $F=\frac{\partial F}{\partial x}=\frac{\partial F}{\partial y}=\frac{\partial F}{\partial z}=0$, where these are the equations from the previous proof.

Again, we know that  $z=1$ and so $y\equiv 0\pmod{p}$, since $\frac{\partial F}{\partial y}=0$. Thus the remaining equations can be rewritten as follows.

\begin{align*}
F(x,y,z)&=&-x^3-ax^2-bx-c, \quad
\frac{\partial F}{\partial x}&=& -3x^2-2ax-b, \quad
\frac{\partial F}{\partial z}&=& -ax^2-2bx-3c.
\end{align*}

Because $2^{k-1}P$ is a singular point, $F(x,y,z)=-f(x)\equiv 0 \pmod{p}$ and $\frac{\partial F}{\partial x}=-f'(x)\equiv 0\pmod{p}$. Thus $f(x)=y^2\equiv 0 \pmod{p}$ in which case $y\equiv 0 \pmod{p}$. 
 So $2\num(Y(2^{k-1}P))\equiv 0 \pmod{p}$, and it follows that $\tau_k F_k\equiv 0 \pmod{p}$. Therefore $\tau_k\equiv 0 \pmod{p}$ or $F_k\equiv 0 \pmod{p}$.

Since $2^{k}P\not\equiv (0:1:0)\pmod{p}$, then $p\nmid e_k$ and so $p\nmid F_k$. Hence $p\vert \tau_k$.

\end{proof}

In addition to looking at $\tau_k$ by examining different aspects of an elliptic curve and its points, we can learn more about $\tau_k$ by considering its parity, which can in turn tell us a little more about elliptic Fermat sequences.
\begin{theorem}
\label{odd Fk}
If $2^{k}P\not\equiv (0:1:0) \pmod{2}$ then $F_k(E,P)$ is odd and $\tau_k$ is even.
\end{theorem}

\begin{proof}
Suppose that $2^{k}P\not\equiv (0:1:0) \pmod{2}$. Thus $2\nmid e_{k}$ and $2\nmid e_{k-1}$ and therefore $F_k$ is odd. Recall that $F_k \tau_k=2(\num(Y(2^{k-1}P))$, so $\tau_k$ must be even.
\end{proof}

The case in which $F_k$ is even is a little more complicated than the previous case, but the parity of $\tau_k$ can still be determined by looking at one extra condition.

\begin{lemma}
\label{even Fk}
If $2^{k}P\equiv (0:1:0) \pmod{2}$ then $F_k(E,P)$ is even. If in addition $2^{k-1}P\equiv (0:1:0)\pmod{2}$, then $\tau_k$ is odd.
\end{lemma}
\begin{proof}

Suppose that $2^{k}P\equiv (0:1:0) \pmod{2}$. This tells us that $2\vert e_k$. 

\textbf{Case I}: Suppose that $2^{k-1}P\not\equiv (0:1:0)\pmod{2}$. Then $2\nmid e_{k-1}$, and $F_k$ is even as $2\vert e_k$ but $2\nmid e_{k-1}$.

\textbf{Case II}:  Suppose that $2^{k-1}P\equiv (0:1:0)\pmod{2}$. Thus $2\vert e_{k-1}$ and $2\nmid n_{k-1}$ and $2\nmid m_{k-1}$. Looking at the duplication formula it can seen that $\gcd(A,B)$ must be odd as $m_{k-1}$ is odd and thus $A$ is odd and $B$ is even. It follows that $\tau_k$ is odd. Thus $\ord_2(e_k)= \ord_2(e_{k-1})+1$. Hence $F_k$ must be even.

Therefore $F_k$ is always even when $2^kP\equiv (0:1:0)\pmod{2}$.
\end{proof}

Combining the past few theorems and lemmas, we can make a nice conclusion about the relationship between the parity of the $\tau_k$ and $F_k$ sequences. Examining this relationship may help us to understand both sequences with more clarity and could lead to a simplified recurrence relation.

\begin{corollary}
The numbers $\tau_k$ and $F_k$ have opposite parity unless $2^{k}P$ reduces to the point at infinity mod $2$ and $2^{k-1}P\not\equiv(0:1:0)\pmod{2}$.
\end{corollary}

While it is nice to know all of these properties, we need to know exactly what $\tau_k$ is in order for the recurrence relations to be useful. In accordance with Theorem \ref{tau alg thm}, we can calculate $|\tau_k|$ for all but finitely many $k$ using the following algorithm:

\begin{enumerate}

\item Find and factor the discriminant $\Delta(E)$.

\item For each prime $p$ such that $p^2|\Delta(E)$, complete the following:

	\begin{enumerate}
	
	\item Find the smallest $\ell \in \mathbb{Z}^+$ such that $\ell P \equiv (0:1:0) \pmod{p}$.
	
	\item \label{tau alg case1} If $\ell$ is a power of 2, then $\ord_p(\tau_k) = 0$ for all $k \geq \ell + 1$.
	
		\begin{enumerate}
		
		\item Move on to the next $p^2|\Delta(E)$.		
		
		\end{enumerate}
	
	\item \label{tau alg case2} If $\ell$ is not a power of 2, then $\ord_p(\tau_k) = \ord_p(2num(Y(2^{k-1}P)))$.
	
		\begin{enumerate}
		
		\item \label{tau alg periodic step}Find some $r \in \mathbb{Z}^+$ such that $rP = \left( \dfrac{m}{e^2}	, \dfrac{n}{e^3} \right)$ with $p^s | e$. Choose $s$ such that either $p^{2s} || \Delta(E)$ or $p^{2s+1} || \Delta(E)$.
		
		\item \label{tau alg lowest ks} Now $\ord_p(Y(tP))$ depends only on $t \mod{r}$. Find all possible values of $2^k \mod{r}$ and note the lowest $k$ which generates each value.
		
		\item Calculate $\ord_p(Y(2^{k-1}P))$ for each $k$ noted in \ref{tau alg lowest ks}. Use this to calculate $\ord_p(\tau_k)$.
		
		\item Move on to the next $p^2|\Delta(E)$.
		
		\end{enumerate}
	
	\end{enumerate}
	
\item \label{tau alg final step} We now know $\ord_p(\tau_k)$ for all (but finitely many, in some cases) $k$ for each $p$ such that $p^2|\Delta(E)$, which are all the $p$ that could divide $\tau_k$. Use this to calculate $|\tau_k|$.

\end{enumerate}

The finitely many $\tau_k$ that this algorithm cannot compute will be at the beginning of the $\tau_k$ sequence, so they can be computed from the definition of $\tau_k$ using finitely many calculations.

Note that sometimes it is difficult to find $r$ in step \ref{tau alg periodic step}, as this step requires being able to add points on the curve, which can not always be done efficiently. If a smaller $s$ is chosen in order to find an $r$, this algorithm can still show that $\tau_k$ is eventually periodic.

Now we will prove that this algorithm is correct. In order to do this, we must first prove the following theorem:

\begin{theorem}
\label{P+Q mod p^k thm}
Let $E: y^2 = x^3 + ax^2 + bx + c$ be an elliptic curve. Assume $Q,R \in E(\mathbb{Q})$ are such that
\begin{equation*}
Q  = (x_1,y_1) = \left( \dfrac{m_1}{e_1^2}, \dfrac{n_1}{e_1^3} \right) \text{, } p\nmid e_1 \text{; }R  = (x_2,y_2) = \left( \dfrac{m_2}{e_2^2}, \dfrac{n_2}{e_2^3} \right) \text{, } p^k\mid \mid e_2.
\end{equation*}
Let $$Q+R = (x_3,y_3) = \left( \dfrac{m_3}{e_3^2}, \dfrac{n_3}{e_3^3} \right).$$
Then
\begin{equation*}
X(Q+R) \equiv X(Q) \pmod{p^k} \text{, } Y(Q+R) \equiv Y(Q) \pmod{p^k}
\end{equation*}
\end{theorem}

The above result follows from Lemma~\ref{modNhom} in the case when $p \nmid 6\Delta(E)$, but in light of the algorithm above, we are primarily interested in the case that $p | \Delta(E)$.

\begin{proof}
From Silverman \cite[p. 58-59]{Silverman}, we know that if we let $\lambda = \dfrac{y_2-y_1}{x_2-x_1}$ and
$v = \dfrac{y_1x_2 - y_2x_1}{x_2-x_1}$, then we have that
\begin{eqnarray*}
x_3 & = & \lambda^2 - a - x_1 - x_2 \\
 & = & \dfrac{ax_2^2 + bx_2 + c - 2y_1y_2 + y_1^2 + 2x_1x_2^2 - x_1^2x_2}{x_2^2 - 2x_1x_2 + x_1^2} - a - x_1.
\end{eqnarray*}

Now since $p^k \mid \mid e_2$, we can let $x_2 = \tilde{x}_2p^{-2k}$ and $y_2 = \tilde{y}_2p^{-3k}$. Plugging this in yields
\begin{equation}
\label{P+Q proof x3 middle}
 x_3 =  \dfrac{a\tilde{x}_2^2 + b\tilde{x}_2p^{2k} + cp^{4k} - 2y_1\tilde{y}_2p^k + y_1^2p^{4k} + 2x_1\tilde{x}_2^2 - x_1^2\tilde{x}_2p^{2k}}{\tilde{x}_2^2 - 2x_1\tilde{x}_2p^{2k} + x_1^2p^{4k}} - a - x_1.
\end{equation}
Reducing mod $p^k$ and mod $p^{2k}$ give us
\begin{gather}
\label{P+Q proof x3 = x1}
x_3 \equiv x_1 \pmod{p^k} \text{ and} \\
\label{P+Q proof x3 mod p2k}
x_3 \equiv x_1 - \dfrac{2y_1\tilde{y}_2p^k}{\tilde{x}_2^2} \pmod{p^{2k}}.
\end{gather}
Now that we have shown that $x_3 \equiv x_1 \pmod{p^k}$, we just need to show that $y_3 \equiv y_1 \pmod{p^k}$.
Since $x_3 \equiv x_1 \pmod{p^k}$, we can write $x_3 = x_1 + rp^k$. And again using $\lambda = \dfrac{y_2-y_1}{x_2-x_1}$ and $v = \dfrac{y_1x_2 - y_2x_1}{x_2-x_1}$, we have that
\begin{eqnarray*}
y_3 & = & -\lambda x_3 - v \\
 & = & \dfrac{-n_1m_1e_2^3 + n_1m_2e_1^2e_2 - n_1e_1^2e_2^3rp^k + n_2e_1^5rp^k}{m_1e_1^3e_2^3 - m_2e_1^5e_2}.
\end{eqnarray*}
Once again, since $p^k \mid \mid e_2$, we can let $e_2 = \tilde{e}_2p^k$. Then
\begin{equation*}
y_3 = \dfrac{-n_1m_1\tilde{e}_2^3p^{2k} + n_1m_2e_1^2\tilde{e}_2 - n_1e_1^2\tilde{e}_2^3rp^{3k} + n_2e_1^5r}{m_1e_1^3\tilde{e}_2^3p^{2k} - m_2e_1^5\tilde{e}_2}.
\end{equation*}
Reducing mod $p^k$ gives us
\begin{equation}
\label{P+Q proof y3 middle}
y_3 \equiv \dfrac{-n_1}{e_1^3} - \dfrac{n_2r}{m_2\tilde{e}_2} \pmod{p^k}.
\end{equation}
Now from equation \eqref{P+Q proof x3 mod p2k}, we know that $r \equiv -\dfrac{2y_1\tilde{y}_2}{\tilde{x}_2^2} \pmod{p^k}$. Simple algebra allows us to see that $r \equiv \dfrac{-2n_1n_2\tilde{e}_2}{m_2^2e_1^3} \pmod{p^k}$.
Plugging this into equation \eqref{P+Q proof y3 middle}, we get
\begin{eqnarray*}
y_3 & \equiv & \dfrac{-n_1}{e_1^3} - \dfrac{n_2}{m_2\tilde{e}_2} \cdot \dfrac{-2n_1n_2\tilde{e}_2}{m_2^2e_1^3} \pmod{p^k} \\
 & \equiv & \dfrac{-n_1}{e_1^3} + \dfrac{2n_1(m_2^3 + am_2^2e_2^2 + bme_2^4 + ce_2^6)}{m_2^3e_1^3} \pmod{p^k}.
\end{eqnarray*}
And since $e_2 \equiv 0 \pmod{p^k}$, we have that
\begin{eqnarray}
y_3 & \equiv & \dfrac{-n_1}{e_1^3} + \dfrac{2n_1m_2^3}{m_2^3e_1^3} \pmod{p^k} \nonumber \\
\label{P+Q proof y3 = y1}
 & \equiv & y_1 \pmod{p^k}.
\end{eqnarray}
\end{proof}

Now we can go on to prove that the algorithm to calculate $\tau_k$ is correct.

\begin{proof}
From Lemma \ref{disc of cubic lemma}, we can conclude that for any $p$ dividing $\tau_k$, we must have $p^2|\Delta(E)$. So we only need to consider primes $p$ which satisfy this condition.
We now break this problem into 2 cases.

\textbf{Case I:} There exists a $d \in \mathbb{Z}^+$ such that $2^dP \equiv (0:1:0) \pmod{p}$.
By Corollary \ref{elliptic fermat order universality prime corollary}, this implies that $p|F_d$. 
It also means that for all sufficiently large $k$ (i.e. $k \geq d$), $2^kP \equiv (0:1:0) \pmod{p}$. This means that $p$ divides the denominator of $X(2^kP)$ and $Y(2^kP)$ and thus $p$ does not divide $num(Y(2^kP))$.
Then since $F_k\tau_k = 2num(Y(2^{k-1}P))$ and we have that $p|F_k$ but $p \nmid num(Y(2^{k-1}P))$, we know that $p \nmid \tau_{k+1}$. So $\ord_p(\tau_k) = 0$ for all $k \geq d+1$.

\textbf{Case II:} $2^kP \not\equiv (0:1:0) \pmod{p}$ for any $k$. 
By Corollary \ref{elliptic fermat order universality prime corollary}, this implies that $p \nmid F_k$ for all $k$. Then since $F_k\tau_k = 2num(Y(2^{k-1}P))$, we have that $\ord_p(\tau_k) = \ord_p(2num(Y(2^{k-1}P)))$. And since $2^kP \not\equiv (0:1:0) \pmod{p}$, we know that $p$ does not divide the denominator of $Y(2^kP)$ for any $k$. Then we have $\ord_p(\tau_k) = \ord_p(2num(Y(2^{k-1}P))) = \ord_p(2Y(2^{k-1}P))$.
Now, we can find some $r \in \mathbb{Z}^+$ such that $rP = \left( \dfrac{m}{e^2}	, \dfrac{n}{e^3} \right)$ with $p^s | e$. Choose $s$ such that either $p^{2s} || \Delta(E)$ or $p^{2s+1} || \Delta(E)$.
Then $rP \equiv (0:1:0) \pmod{p^s}$. Using Theorem \ref{P+Q mod p^k thm}, we can see that $jP + rP \equiv jP \pmod{p^s}$ and conclude that $\ord_p(Y(tP))$ depends only on $t \mod{r}$. Then, since $2^k \mod{r}$ will repeat, we can use a finite number of calculations to determine $\ord_p(Y(2^kP))$ for all $k \geq 1$.

Now all that is left to show is that the ``if" statements in steps \ref{tau alg case1} and \ref{tau alg case2} correspond to the correct case.
It is obvious that if $\ell$ is a power of 2 (as required for step \ref{tau alg case1}), we are in Case I, and step \ref{tau alg case1} corresponds to this case.
Now, if $\ell \in \mathbb{Z}^+$ such that $\ell P \equiv (0:1:0) \pmod{p}$ is minimal but not a power of 2 (as required for step \ref{tau alg case2}), then any other $\ell'$ satisfying $\ell'P \equiv (0:1:0) \pmod{p}$ will be a multiple of $\ell$ and thus will not be a power of 2. Then we are in Case II, which corresponds to step \ref{tau alg case2}.

\end{proof}

\section{Primality}
\label{primality sec}

In this section, we will discuss a few theorems about the primality of the elliptic Fermat numbers. Our first theorem on this topic focuses on sequences for which the denominators of the coordinates of $P$ are even.

\begin{theorem}
\label{primality ek even}
If $2|e_t$ for some $t$, then for all $k \geq t$, either $F_k = 2$ or $F_k$ is composite.
\end{theorem}

\begin{proof}
Suppose that $2^tP = \left( \dfrac{m_t}{e_{t}^2}, \dfrac{n_t}{e_{t}^3} \right)$ and $2|e_t$. This tells us that $2^tP\equiv (0:1:0)\pmod{2}$. Therefore $F_k$ is even for all $k\geq t$. In which case $F_k=2$ or $F_k$ is a multiple or $2$ greater than $2$ and is therefore composite. 
\end{proof}

From this theorem, we also have the following corollary:

\begin{corollary}
If $2|e_t$ for some $t$, then for all $k \geq \ell$ for some sufficiently large $\ell$, $F_k$ is composite.
\end{corollary}

To prove the corollary, we need only show that $F_k\ne 2$ for all sufficiently large $k$.

\begin{proof}
We know that $F_k = \dfrac{2n_{k-1}}{\tau_k}$. Then in order for $F_k$ to equal 2, we must have that $n_{k-1} = \tau_k$. But since $\tau_k$ is periodic for all $k \geq \ell$ for some finite $\ell$, and $n_k$ is not, we know that $F_k \neq 2$ for all $k \geq \ell$ for some finite $\ell$.
\end{proof}

The case in which the denominator of $2^kP$ is always odd is trickier, and in fact we have not come up with a theorem covering all such elliptic Fermat sequences. The theorem that we do have requires a few lemmas.

\begin{lemma}
\label{only one int pt}
Assume that $E(\mathbb{Q}) \cong \mathbb{Z} \times \mathbb{Z}/2\mathbb{Z}$ and $E(\mathbb{Q}) = \langle P,T \rangle$, where $P$ is a generator of $E(\mathbb{Q})$ and $T$ is a rational point of order $2$. Assume that:
\begin{enumerate}[(i)]
\item $E$ has an egg.
\item $T$ is on the egg.
\item $T$ is the only integral point on the egg.
\item $P$ is not integral.
\end{enumerate}
Then $T$ is the only integral point on $E$.
\end{lemma}

\begin{proof}
Every point in $E(\mathbb{Q})$ is of the form $mP$ or $mP+T$. If $P$ is on the nose, then  we have that for any $m\ne 0$, $mP$ is on the nose, and $mP$ is not integral because $P$ is not integral. We also have that $mP+T$ is on the egg and thus is not integral because $T$ is the only integral point on the egg by assumption. If $P$ is on the egg, then let $P' = P+T$. Then $P'$ is on the nose, and the proof is the same as before.
\end{proof}

\begin{lemma}
\label{gcd(m,b) = 1}
Let elliptic curve $E$ be of the form $y^2 = x^3 + ax^2 + bx$ and suppose $\gcd(m_0,b) = 1$. Then $\gcd(m_k,b) = 1$ for all $k$.
\end{lemma}

\begin{proof}
We use induction. The base case $\gcd(m_0,b) = 1$ is true by assumption. Now assume that $\gcd(m_{k-1},b) = 1$.
Since $c = 0$, from our recurrence relations, we can see that
\begin{equation*}
m_k = \dfrac{m_{k-1}^4 - 2bm_{k-1}^2e_{k-1}^4 + b^2e_{k-1}^8}{\tau_k^2}.
\end{equation*}
Now since $b$ divides the $-2bm_{k-1}^2e_{k-1}^4$ and $b^2e_{k-1}^8$ terms in the numerator but is coprime to the $m_{k-1}^4$ term, $b$ is coprime to the numerator. Dividing by $\tau_k^2$ will not change this. Thus $\gcd(m_k,b) = 1$ for all $k$.
\end{proof}

With these two lemmas, we can now prove Theorem \ref{primality ek odd}.

Note that the condition that $2 \nmid e_k$ for all $k$ can be checked with finitely many calculations by looking at the Tamagawa number at $2$ for the curve $E$. If the curve has additive reduction and $P\not\in E_0(\mathbb{Q}_p)$, then the Tamagawa number can only be $1$, $2$, $3$, or $4$. The condition holds when the Tamagawa number at $2$ is $3$ because $\ell P$ is a singular point mod $2$ unless $3\vert \ell$. Also note that the condition that $x^4 + ax^2y^2 + by^4 = 1$ has no integer solutions where $y \not \in \{0,\pm 1 \}$ can also be checked with finitely many calculations, as this is a Thue equation. Such an
equation has finitely many solutions (by \cite{Thue}), and the solutions
can be found effectively (see \cite{TzanakisWeger}). We will now go on to prove the theorem.

\begin{proof}[Proof of Theorem \ref{primality ek odd}]
Without loss of generality, let $c = 0$ and let $T = (0,0)$. (If not, we can easily shift the curve so that this is true.)
Let $2^{k-1}P=\left(\dfrac{m_{k-1}}{e_{k-1}^2},\dfrac{n_{k-1}}{e_{k-1}^3}\right)$, and let $2^{k-1}P+T=\left(\dfrac{m_{T}}{e_{T}^2},\dfrac{n_{T}}{e_{T}^3}\right)$.

Using the formulas for adding points given by Silverman \cite[p. 58-59]{Silverman}, we can see that
\begin{gather}
\label{prim proof X2PT}
X(2^{k-1}P + T) = \dfrac{be_{k-1}^2}{m_{k-1}}, \\
\label{prim proof Y2PT not lowest terms}
Y(2^{k-1}P + T) = \dfrac{-bn_{k-1}e_{k-1}}{m_{k-1}^2}. \nonumber
\end{gather}

By the assumption that $\gcd(b,m_0) = 1$ and by Lemma \ref{gcd(m,b) = 1}, we know that $\gcd(b,m_{k-1}) = 1$. And since $\gcd(m_{k-1},e_{k-1}) = 1$, equation \eqref{prim proof X2PT} must be in lowest terms. Then $e_T = \sqrt{|m_{k-1}|}$, so we can set up the following equation:

\begin{equation*}
\dfrac{-bn_{k-1}e_{k-1}}{m_{k-1}^2} = \dfrac{n_T}{\sqrt{|m_{k-1}|}^3}.
\end{equation*}

Solving for $n_T$ yields

\begin{equation}
\label{prim proof nT}
n_T = \dfrac{-bn_{k-1}e_{k-1}}{\sqrt{|m_{k-1}|}}.
\end{equation}

Now, notice that $2(2^{k-1}P)=2^kP$ and also $2(2^{k-1}P+T)=2^kP$ as $T$ has order 2 in which case $2T$ is the point at infinity. Then

\begin{eqnarray*}
\denom(2(2^{k-1}P)) & = & \denom(2(2^{k-1}P+T)) \\
\dfrac{4n_{k-1}^2e_{k-1}^2}{\tau_k^2} & = & \dfrac{4n_T^2e_T^2}{\tau_T^2} \\
\dfrac{2n_{k-1}e_{k-1}}{\tau_k} & = & \dfrac{2n_Te_T}{\tau_T}.
\end{eqnarray*}

Solving for $\tau_T$ yields

\begin{equation*}
\tau_T = \dfrac{\tau_kn_Te_T}{n_{k-1}e_{k-1}}.
\end{equation*}

Plugging in $n_T = \dfrac{-bn_{k-1}e_{k-1}}{\sqrt{m_{k-1}}}$, $e_T = \sqrt{|m_{k-1}|}$, and $|\tau_k| = 2$ gives us 

\begin{equation*}
\label{prim proof tauT}
|\tau_T| = \left|\dfrac{-2bn_{k-1}e_{k-1}\sqrt{|m_{k-1}|}}{n_{k-1}e_{k-1}\sqrt{|m_{k-1}|}}\right| = 2|b|.
\end{equation*}

Now, the duplication formula tells us that

\begin{equation*}
\denom(2(2^{k-1}P+T)) = \dfrac{4n_T^2e_T^2}{\tau_T^2} = \dfrac{n_T^2e_T^2}{b^2}.
\end{equation*}

And since $\denom(2(2^{k-1}P))=\denom(2^kP)$, we have that $F_k = \left| \dfrac{n_Te_T}{be_{k-1}} \right|$.

Note that if $p$ is a prime and $p|e_{k-1}$ then $2^{k-1}P\equiv(0:1:0)\pmod{p}$ in which case $2^{k-1}P+T\equiv T\pmod{p}$. And since $T$ is not the point at infinity, $2^{k-1}P+T\not\equiv (0:1:0)\pmod{p}$. Therefore $p\nmid e_T$. Hence $\gcd(e_{k-1},e_T)=1$.
And since $\gcd(m_{k-1},b) = 1$ and $e_T = \sqrt{|m_{k-1}|}$, we have that $\gcd(e_T, b) = 1$.
Thus $F_k = \left| \dfrac{-n_T}{be_{k-1}} \right| \cdot e_T$. Note that $e_T \ne 1$ as there is only one integral point on this curve. Therefore $F_k$ is composite as long as $\dfrac{n_T}{be_{k-1}} \ne \pm 1$. Plugging in \eqref{prim proof nT} for $n_T$ yields

\begin{eqnarray*}
\dfrac{n_T}{be_{k-1}} & = & \dfrac{\left( \frac{-bn_{k-1}e_{k-1}}{\sqrt{|m_{k-1}|}} \right)}{be_{k-1}} \\
 & = & \dfrac{n_{k-1}}{\sqrt{|m_{k-1}|}}
\end{eqnarray*}

Thus $\dfrac{n_T}{be_{k-1}} = \pm 1$ if and only if $n_{k-1} = \pm \sqrt{|m_{k-1}|}$.

We now proceed by contradiction. Assume that $n_{k-1} = \pm \sqrt{|m_{k-1}|}$. Then $n_{k-1}^2 = |m_{k-1}|$. Plugging this into the equation for the cubic yields $|m_{k-1}|=m_{k-1}^3+am_{k-1}^2e_{k-1}^2+bm_{k-1}e_{k-1}^4$. And thus $m_{k-1}^2+am_{k-1}e_{k-1}^2+be_{k-1}^4= \pm 1$. But by assumption, this equation has no solutions where $e_{k-1} \not \in \{0, \pm 1 \}$. Therefore $F_k$ is composite for all $k \geq 1$.

\end{proof}

\section{Growth Rate}
\label{growth rate sec}

In this section, we will discuss the growth rate of the elliptic Fermat numbers. In order to do so, we need a few more tools. The first new definition we need is the \emph{height} of a point.

\begin{definition}
The \emph{height} of a point $P = \left( \dfrac{m}{e^2}, \dfrac{n}{e^3} \right)$ on an elliptic curve is defined as
\begin{equation*}
h(P) = \log(\max(|m|, e^2)).
\end{equation*}
\end{definition}

The height of a point gives us a way to express how ``complicated" the coordinates of the point are. We also need to make use of the \emph{canonical height}.

\begin{definition}
The \emph{canonical height} of a point $P$ on an elliptic curve is defined as
\begin{equation*}
\hat{h}(P) = \lim_{k \to \infty} \dfrac{h(2^kP)}{4^k}.
\end{equation*}
\end{definition}

Interestingly, if we let $\ell P = \left( \dfrac{A_\ell}{C_\ell^2}, \ast \right)$ with $\gcd(A_\ell, C_\ell) = 1$, then $\lim\limits_{\ell \to \infty} \dfrac{\log(C_\ell^2)}{\ell^2} = \lim\limits_{\ell \to \infty} \dfrac{|A_\ell|}{\ell^2} = \hat{h}(P)$ \cite[p. 250]{Silverman}. This allows us to derive Theorem \ref{growth rate thm}.

Note that this theorem can also be stated as $F_k \approx e^{4^k \cdot \frac{3}{8} \hat{h}(P)}$. So the elliptic Fermat sequences grow doubly exponentially, like the classic Fermat sequence, albeit much more quickly. The proof is as follows:

\begin{proof}
\begin{eqnarray*}
\lim_{k \to \infty} \dfrac{\log(F_k(E,P))}{4^k} & = & \lim_{k \to \infty} \dfrac{\log(\frac{e_k}{e_{k-1}})}{4^k} \\
 & = & \lim_{k \to \infty} \dfrac{\frac{1}{2}\log(e_k^2)}{4^k} - \lim_{k \to \infty} \dfrac{\frac{1}{2}\log(e_{k-1}^2)}{4 \cdot 4^{k-1}} \\
 & = & \dfrac{1}{2} \lim_{k \to \infty} \dfrac{\log(e_k^2)}{4^k} - \dfrac{1}{8} \lim_{k \to \infty} \dfrac{\log(e_{k-1}^2)}{4^{k-1}} \\
 & = & \dfrac{1}{2}\hat{h}(P) - \dfrac{1}{8}\hat{h}(P) \\
 & = & \dfrac{3}{8}\hat{h}(P).
\end{eqnarray*}
\end{proof}

\section{$y^2 = x^3 - 2x$}
\label{specialcurve}

In this section, we apply the hitherto developed theory of elliptic Fermat numbers to examine properties of the curve $E: y^2 = x^3 -2x$ and the point $P = (2,2) \in E(\Q)$. 

We begin with some remarks on $E$ and the point $P$. Recall that $E$ is equipped with complex multiplication and so Proposition \ref{specialorder} gives a formula for $|E(\F_p)|$ for all $p$. Elliptic curves
with complex multiplication are the key to the Atkin-Goldwasser-Kilian-Morain
elliptic curve primality proving algorithm, and elliptic curve algorithms
to prove primality of Fermat numbers and other special sequences have
been considered previously in \cite{Gross}, \cite{DenommeSavin}, \cite{Tsumura},
and most recently \cite{Silverbergetal}. The last remark we make is about the elliptic Fermat sequence $\{F_n(E,P)\}$ and the appearance of Fermat and Mersenne primes, primes of the form $2^p-1$ for a prime $p$, in the factorization of $F_k(E,P)$.

\begin{center}
\begin{tabular}[c]{|l| l|}
\hline
$n$ & $F_n(E,P)$ \\
\hline
0 & 1 \\
\hline
1 & 2 \\
\hline
2 & $ 2\cdot \mathbf{3}\cdot \mathbf{7}$ \\
\hline
3 & $ 2\cdot \mathbf{31} \cdot 113 \cdot \mathbf{257}$ \\
\hline
4 & $2 \cdot 2113 \cdot 2593 \cdot 46271 \cdot 101281 \cdot 623013889$\\
\hline
5 & $2 \cdot \mathbf{127} \cdot \mathbf{65537} \cdot 33303551 \cdot 70639871 \cdot 364024274689 \cdot \cdots \cdot 676209479362440577$\\
\hline
\end{tabular}
\end{center}

The table above provides a factorization of the first 6 elliptic Fermat numbers for $E$ at $P$, with known Fermat and Mersenne primes in bold. In fact, every odd prime factor dividing $F_n(E,P)$ for $n \geq 2$ will have a congruence that is either Mersenne-like or Fermat-like. We now present the proof of Theorem \ref{cong}, beginning with the congruence result for a prime divisor $p \equiv -1 \pmod 4$, which yields a tidy Mersenne-like congruence.

\begin{proof}[Proof of Theorem \ref{cong} for $p \equiv 3 \pmod 4$]
By Theorem \ref{elliptic fermat order universality thm}, $p \mid F_n(E,P)$ tells us that $P$ has order $2^n$ in $E(\F_p)$. Then by Lagrange's theorem and Proposition \ref{specialorder}, $2^n \vert |E(\F_p)| = p+1$, and so $p \equiv -1 \pmod {2^n}$.
\end{proof}

Proving the congruence in the case of a prime divisor of an elliptic Fermat number congruent to 1 modulo 4 will require multiple steps. We will eventually show that such a prime divisor of $F_n(E,P)$ is congruent to $1$ modulo $2^n$, but we begin by showing an initial congruence result:

\begin{lemma}
Let $E: y^2 = x^3 - 2x$ be an elliptic curve, $P = (2,2)$ a point of infinite order and $F_n(E,P)$ the $n$th elliptic Fermat number associated to $E$ at the point $P$. Then for any odd prime divisor $p \equiv 1 \pmod 4$ of $F_n(E,P)$, $n \geq 3$, 
$ p \equiv  1 \pmod{\max(2^{\floor{n/2}}, 8)}.$
\end{lemma}

\begin{proof}
If $p \equiv 1 \pmod 4$, then $p = a^2 + b^2$ where $a \equiv -1 \pmod 4$. Recall that in this situation, the value of $|E(\F_p)|$ depends on the quartic character of 2 modulo $p$. Let us first consider the case where 2 is a fourth power. Then $|E(\F_p)| = p+1 - 2a$. 

Like the proof of the previous theorem, we use Lagrange's theorem to show that $2^n \mid E(\F_p) = a^2+b^2+1- 2a = (a-1)^2 + b^2$. So $(a-1)^2 + b^2 \equiv 0 \pmod {2^n}$. Then $a-1 \equiv b \equiv 0 \pmod{ 2^{\lfloor{n/2}\rfloor}}$, giving $p = a^2 + b^2 = (a-1)^2 + 2a - 1 b^2 \equiv 1 \pmod { 2^{\lfloor{n/2}\rfloor}}$. A symmetric argument follows when 2 is a quadratic residue but not a fourth power. In this situation we arrive at the equation $(a+1)^2 + b^2 \equiv 0 \pmod{2^n}$, however the result is precisely the same.

To conclude, we rule out the case where 2 is not a quadratic residue modulo $p$. This would imply $|E(\F_p)| = p+1 \pm 2b$. The same algebraic manipulation leads to a similar situation where $a^2 + (b \mp 1)^2 \equiv 0 \pmod{2^n}$, but this means $b \equiv \pm 1 \pmod {2^{\floor{n/2}}}$, however $b$ is the even part of the two-square representation of $p$. So it cannot be the case that 2 is not a quadratic residue modulo 8, which happens only when $p \equiv 5 \pmod 8$.
\end{proof}

Because of the lemma, we have $p \equiv 1 \pmod 8$, and so we can make sense of $\sqrt 2$ and $i$ modulo $p$. We now define the recklessly-notated action $i$ on $E(\F_p)$ as $i(x,y) \mapsto (-x,iy)$, where the point $(-x,iy)$ uses $i$ as the square root of $-1$ modulo $p$.

This action makes $E(\F_p)$ into a $\Z[i]$-module. We will prove one last lemma concerning the action of $(1+i)$ before moving on to the full congruence.

\begin{lemma} 
Let $E: y^2 = x^3 - 2x$ be an elliptic curve, $P = (2,2)$ a point of infinite order and $F_n(E,P)$ the $n$th elliptic Fermat number associated to $E$ at the point $P$. Then for any odd prime factor $p \equiv 1 \pmod 4$ of $F_n(E,P)$, $n \geq 3$, we have that $(1+i)^{2n}P = 0$ and $(1+i)^{2n-2}P \neq 0 $. 
\end{lemma}

\begin{proof}
Note that $(1+i)^k P = 2^ki^kP$. Recall that $P$ has order $2^n$, so $(1+i)^{2n} P = (2i)^nP = i^n(2^n P) = i^n\cdot 0 = 0$. It suffices to show that $(1+i)^{x} P \neq 0$ for $x \leq 2n-2$. Suppose not, and $(1+i)^{x}P = 0$. Then certainly $(1+i)^{2n-2} = i^{n-1} 2^{n-1} P = 0$.
The action of $i^{n-1}$ makes no difference on the identity. This implies that $2^{n-1} P = 0$, contradicting order universality since $P$ has order $2^n$.
\end{proof}

With this last lemma proven, we are ready to introduce the Fermat-like congruence in full regalia and finish Theorem \ref{cong}.

\begin{proof}[Proof of Theorem \ref{cong} for $p \equiv 1 \pmod 4$]
As a consequence of the above lemma, we have that either $(1+i)^{2n}P = 0$ or $(1+i)^{2n-1}P = 0$. We are able to bolster the $2n-1$ case by introducing a new point $Q = (-i(\sqrt 2 - 2), (2-2i)(\sqrt 2 - 1))$. It is routine point addition to see that $(1+i)Q = (2,2) = P$.  In either case we have that $(1+i)^{2n+1}Q = 0$ and $(1+i)^{2n-1}Q \neq 0$. 

Consider the $\Z[i]$-module homomorphism $\phi: \Z[i] \to E(\F_p)$ given by $\phi(x) = xQ$. The image of $\phi$ is $\Z[i]Q = \{(a+bi)Q\mid a,b \in \Z\}$, the orbit of $\Z[i]$ on $Q$. By the first isomorphism theorem, $\Z[i]Q$ is isomorphic to $\Z[i] / \ker(\phi)$. Since $(1+i)^{2n-1} \not\in \ker(\phi)$ and $(1+i)^{2n+1} \in \ker(\phi)$, and $(1+i)$ is an irreducible ideal in $\Z[i]$, then the kernel is either the ideal $((1+i)^{2n})$ or $((1+i)^{2n+1})$, hence $\Z[i] / \ker(\phi)$ is a group of size $2^k$ where $k = 2n$ or $k = 2n+1$.

Like the previous congruence results, we use Lagrange's theorem to assert $2^k \mid |E(\F_p)|$ and through the same reasoning as before, we arrive at $p \equiv 1 \pmod {2^{\floor{k/2}} = 2^{n}}$.
\end{proof}

We now present the proofs of Theorems \ref{congversef} and \ref{congversem}, which give us information about sufficiently large Fermat and Mersenne primes dividing the elliptic Fermat sequence $\{F_n(E,P)\}$. First, we provide two lemmas.

\begin{lemma}\label{uniroot}
Let $p \equiv \pm 1 \pmod{2^n}$ be an odd prime. Let $\zeta_{\ell}$ denote a primitive $\ell$th root of unity in some extension of $\F_p$. Then $\zeta_{2^k} + \zeta_{2^k}^{-1}$ exists in $\F_p$ for all $k \leq n$.
\end{lemma}

\begin{proof}
If $p \equiv 1 \pmod {2^k}$, then clearly there is a primitive $2^k$th root of unity in $\F_p$. 

If $p \equiv 3 \pmod 4$, then we employ methods from Galois theory. First, because $p \equiv -1 \pmod{2^k}$, then $p^2 \equiv 1 \pmod{2^k}$. Then there is a primitive $2^k$th root of unity in $\F_{p^2}$. Then we have that $\alpha = \zeta_{2^k} + \zeta_{2^k}^{-1}$ is in $\F_p$ if and only if $\sigma(\alpha) = \alpha$, where $\sigma(x) = x^p$ the Frobenius endomorphism.

This says that $\alpha \in \F_p$ if and only if $\alpha^p = (\zeta_{2^k} + \zeta_{2^k}^{-1})^p = \zeta_{2^k}^p + \zeta_{2^k}^{-p} = \zeta_{2^k} + \zeta_{2^k}^{-1}$. We may write this equality as $\zeta_{2^k}^{2p} + \zeta_{2^k}^{p+1} + \zeta_{2^k}^{-p+1} + 1 = 0$. This factors into $(\zeta_{2^k}^p - \zeta_{2^k})(\zeta_{2^k}^p - \zeta_{2^k}^{-1}) = 0$. Then the equality holds if and only if $\zeta_{2^k}^p = \zeta_{2^k}$, meaning $p \equiv 1 \pmod {2^k}$, or $\zeta_{2^k}^p = \zeta_{2^k}^{-1}$, hence $p \equiv -1 \pmod{2^k}$.
\end{proof}

\begin{lemma}\label{Qexist}
Let $p$ be a Fermat or Mersenne prime that is at least 31. Then there exists a $Q \in E(\F_p)$ such that $2Q = P$.
\end{lemma}

\begin{proof}
From Silverman and Tate \cite[p. 76]{SilvermanTate}, for $E$ we have its isogenous curve $E': y^2 = x^3 + 8x$ and two homomorphisms, $\phi: E \to E'$ and $\psi: E' \to E$ given by:

$$\phi(x,y) = \begin{cases}
\left(\frac{y^2}{x^2}, \frac{y(x^2 + 2)}{ x^2}\right) & \text{if } (x,y) \neq (0:0:1), (0:1:0)  \\
(0:1:0), & \text{otherwise},\\
\end{cases}
$$

$$\psi(x,y) = \begin{cases}
\left(\frac{y^2}{4x^2}, \frac{y(x^2 - 8)}{ 8x^2}\right) & \text{if } (x,y) \neq (0:0:1), (0:1:0)  \\
(0:1:0), & \text{otherwise}.\\
\end{cases}
$$

The maps hold the special property $\phi \circ \psi (S) = 2S$. The advantage of this framework is that we are able to break point-halving, a degree 4 affair, into solving two degree 2 problems. Another fact from Silverman and Tate \cite[p. 85]{SilvermanTate} is that $P = (x,y) \in \psi(E'(\Q))$ if and only if $x$ is a square.

We now use this to show there is a $Q \in E(\F_p)$ such that $2Q = P$. For brevity, let $z = \sqrt{2+\sqrt{2}}$. and we define the following ascending chain of fields: $\Q, K = \Q\left(\sqrt2\right)$ and $L = K(z)$. Here $K$ is the minimal subfield where $P$ has a $\psi$ preimage $Q_1$ in $E'$, and $L$ is the minimal subfield where that preimage has its own $\phi$ preimage $Q$ in $E$. It is a quick check in Magma to verify that for $E(L)$, $P$ is divisible by 2. It then remains to verify that the elements $\sqrt{2}$ and $z = \sqrt{2+\sqrt{2}}$ are in $\F_p$.

First, we have that since $2$ has order $p$, which is odd, then there exists $h_k \in (\F_p)^\times$ such that $(h_k)^{2^k} = 2$. So any 2-power root of 2 is sure to exist.

For $z = \sqrt{2+\sqrt{2}}$ itself, we use Lemma \ref{uniroot} and $p \equiv \pm1 \pmod {16}$ to show that we have an element $z = \zeta_{16} + \zeta_{16}^{-1} \in \F_{p}$, so we have all the necessary elements of $L$ in $E(\F_p)$ to show there exists a $Q \in E(\F_p)$ such that $2Q = P$.

\end{proof}

These two lemmas will allow us to sharpen the threshold to search for Fermat and Mersenne primes in the elliptic Fermat sequence. We now prove Theorem \ref{congversef}.

\begin{proof}
First, it is a quick computation in Magma to verify that for $p = 5,17$, $P$ does not have a 2-power order in $E(\F_p)$, and so by Corollary \ref{elliptic fermat order universality prime corollary}, 5 and 17 do not divide any elliptic Fermat number generated by $P$.

We rely on Proposition \ref{specialorder} and Lagrange's theorem. For a classical Fermat prime $F_n \neq 5, 17$, we have that $2$ is a fourth power in $\Z/F_n\Z$.
We can see this because for a generator $g$ of $\Z/F_n\Z$, we have that $2= g^k$, additionally, we have that $g^{p-1} = g^{2^{2^n}} = 1$. We will show that $k \equiv 0 \pmod 4$. This is because $2$ has order $2^{n+1} \in \Z/F_n\Z$, and so $2^{2^{n+1}} = (g^{k})^{2^{n+1}} = 1$. Therefore, $2^{2^n} \mid k(2^{n+1})$, finally giving $2^{2^n - n-1} \mid k$, which is a multiple of 4 for $n \geq 3$.

Since $2$ is a fourth power in $\F_p$, we know that $E : y^2 = x^3 - 2x$ is isomorphic to the curve $E' : y^2 = x^3 - x$. From Denomme and Savin \cite{DenommeSavin}, we also have that $E'(\F_p) \cong \Z[i]/(1+i)^{2^n}$. Moreover, $\Z[i]/(1+i)^{2^n} = \Z[i]/2^{2^{n-1}} \cong (\Z/2^{2^{n-1}}\Z) \times (\Z/2^{2^{n-1}}\Z)$, from which we can deduce that $E(\F_p) \cong (\Z/2^{2^{n-1}}\Z) \times (\Z/2^{2^{n-1}}\Z)$. 
Thus the order of $P$ is a divisor of $2^{2^{n-1}}$.

By Lemma \ref{Qexist}, we know there exists some $Q \in E(\F_p)$ such that $2Q = P$. In light of this we can tighten this initial upper bound by noting that all elements have order dividing $2^{2^{n-1}}$, and so $2^{2^{n-1}-1}P = 2^{2^{n-1}-1}(2Q) = 2^{2^{n-1}}Q = 0$. We conclude that P has order dividing $2^{2^{n-1}-1}$ and so $p$ must divide $F_k(E,P)$ for some $k \leq 2^{{n-1}}-1$ by Corollary \ref{elliptic fermat order universality prime corollary}.

\end{proof}

It remains to discuss the appearance of a Mersenne prime in the elliptic Fermat sequence. We prove Theorem \ref{congversem}.

\begin{proof}
The method we take to show this bound begins with the fact that $|E(\F_p)| = p+1 = 2^q$. Additionally, we have that $E(\F_p) \cong \Z/m\Z \times \Z/mn\Z$, where $p \equiv 1 \pmod m$. Combining this with $p \equiv -1 \pmod {2^q}$ we have that $E(\F_p) \cong\Z/2\Z \times \Z/2^{p-1}\Z$. So the order of any point in $E(\F_p)$ must divide $2^{p-1}$. It suffices to exhibit a point $R$ such that $4R = P$, so that $2^{p-3}P = 2^{p-3}2^{2}R = 2^{p-1}R = 0$.

Continuing the methodology first used in the proof of Lemma \ref{Qexist}, we will show that such an $R \in E(\F_p)$ so that $2R = Q$, where $Q \in E(L)$ is the point found in Lemma \ref{Qexist} . To this, we extend the fields from Lemma \ref{Qexist} and create $M = L\left(\sqrt{z(2+z)}\right)$ and  $N = M \left(\sqrt{\sqrt{2}(z-1)}\right)$. Again, one may check in Magma that indeed $P$ is divisible by 4 in $E(N)$, so we just need to check for the existence of necessary elements.

We have already shown there is an element $z$ such that $z^2 = 2 + \sqrt{2}$, but we further assert that in $\F_p$, $2+\sqrt{2}$ has odd order, and thus all 2-power roots exist. This is quick to see because $ (2+\sqrt{2})^{(p-1)/2} = (z^2)^{(p-1)/2} = z^{p-1} = 1$.

We now find $\sqrt{z(2+z)}$, which amounts to finding a square root of $z$ and $2+z$. By the above, we already have a square root of $z$, so we just need to show the existence of the square root of $2+z$. This is simple if we let $w = \zeta_{32} + \zeta_{32}^{-1}$ in $\F_p$, which we know to exist if $p \equiv -1 \pmod {32}$. Then $w^2 = 2+z$.

It remains to find $\sqrt{\sqrt{2}(z-1)}$. Again it suffices to just find a square root of $z-1$. To show such a root exists, consider $(z-1)(-z-1) = -z^2 + 1 = 1 - \sqrt{2} = (-1)(1+\sqrt{2})$. Note that $z = \sqrt[4]{2}\sqrt{(1+\sqrt{2})}$, and that $1+\sqrt{2}$ is a square because $\sqrt[4]{2}$ and $z$ are squares, but $-1$ is not a square modulo $p$ since $p \equiv -1 \pmod 4$, so $(z-1)(-z-1)$ is not a square. This implies that exactly one of $(z-1)$ and $(-z-1)$ is a square. So we choose the appropriate $z'$ such that $z' - 1 $ is a square and we are done.

Since all adjoined elements exist in $\F_p$, we are good to construct points $R$ such that $4R = 2Q = P$. Similar to Theorem \ref{congversef}, this implies that we can tighten the condition that $|P| \mid 2^{p-1}$ further by $|P| \mid 2^{p-3}$, and so by Corollary \ref{elliptic fermat order universality prime corollary}, $p$ must divide $F_k(E,P)$ for some $k \leq p-3$.
\end{proof}

\bibliographystyle{ieeetr}
\bibliography{PAPER_bib}

\begin{thebibliography}{10}

\bibitem{FermatBook}
P.~Fermat, {\em \OE uvres de {P}ierre {F}ermat. {II}}.
\newblock Imprimerie Gauthier-Villar Et Fils, 1894.
\newblock Translated by Paul Tannery.

\bibitem{DummitFoote}
D.~S. Dummit and R.~M. Foote, {\em Abstract algebra}.
\newblock John Wiley \& Sons, Inc., Hoboken, NJ, third~ed., 2004.

\bibitem{Lenstra2}
H.~W. Lenstra, Jr., ``Factoring integers with elliptic curves,'' {\em Ann. of
  Math. (2)}, vol.~126, no.~3, pp.~649--673, 1987.

\bibitem{AtkinMorain}
A.~O.~L. Atkin and F.~Morain, ``Elliptic curves and primality proving,'' {\em
  Math. Comp.}, vol.~61, no.~203, pp.~29--68, 1993.

\bibitem{Koblitz}
N.~Koblitz, ``Elliptic curve cryptosystems,'' {\em Math. Comp.}, vol.~48,
  no.~177, pp.~203--209, 1987.

\bibitem{Miller}
V.~S. Miller, ``Use of elliptic curves in cryptography,'' in {\em Advances in
  cryptology---{CRYPTO} '85 ({S}anta {B}arbara, {C}alif., 1985)}, vol.~218 of
  {\em Lecture Notes in Comput. Sci.}, pp.~417--426, Springer, Berlin, 1986.

\bibitem{Wiles}
A.~Wiles, ``Modular elliptic curves and {F}ermat's last theorem,'' {\em Ann. of
  Math. (2)}, vol.~141, no.~3, pp.~443--551, 1995.

\bibitem{PSS}
B.~Poonen, E.~F. Schaefer, and M.~Stoll, ``Twists of {$X(7)$} and primitive
  solutions to {$x^2+y^3=z^7$},'' {\em Duke Math. J.}, vol.~137, no.~1,
  pp.~103--158, 2007.

\bibitem{Silverman2}
J.~H. Silverman, {\em Advanced topics in the arithmetic of elliptic curves},
  vol.~151 of {\em Graduate Texts in Mathematics}.
\newblock Springer-Verlag, New York, 1994.

\bibitem{Magma}
W.~Bosma, J.~Cannon, and C.~Playoust, ``The {M}agma algebra system. {I}. {T}he
  user language,'' {\em J. Symbolic Comput.}, vol.~24, no.~3-4, pp.~235--265,
  1997.
\newblock Computational algebra and number theory (London, 1993).

\bibitem{sagemath}
T.~S. Developers, {\em {S}ageMath, the {S}age {M}athematics {S}oftware {S}ystem
  ({V}ersion 7.5.1)}, 2017.
\newblock {\tt http://www.sagemath.org}.

\bibitem{SilvermanTate}
J.~H. Silverman and J.~Tate, {\em Rational points on elliptic curves}.
\newblock Undergraduate Texts in Mathematics, Springer-Verlag, New York, 1992.

\bibitem{Cohen}
H.~Cohen, {\em Number theory. {V}ol. {I}. {T}ools and {D}iophantine equations},
  vol.~239 of {\em Graduate Texts in Mathematics}.
\newblock Springer, New York, 2007.

\bibitem{Lenstra}
H.~W. Lenstra, Jr., ``Elliptic curves and number-theoretic algorithms,'' in
  {\em Proceedings of the {I}nternational {C}ongress of {M}athematicians,
  {V}ol. 1, 2 ({B}erkeley, {C}alif., 1986)}, pp.~99--120, Amer. Math. Soc.,
  Providence, RI, 1987.

\bibitem{LangeRuppert}
H.~Lange and W.~Ruppert, ``Complete systems of addition laws on abelian
  varieties,'' {\em Invent. Math.}, vol.~79, no.~3, pp.~603--610, 1985.

\bibitem{Silverman}
J.~H. Silverman, {\em The arithmetic of elliptic curves}, vol.~106 of {\em
  Graduate Texts in Mathematics}.
\newblock Springer-Verlag, New York, 1986.

\bibitem{Thue}
A.~Thue, ``\"uber {A}nn\"aherungswerte algebraischer {Z}ahlen,'' {\em J. Reine
  Angew. Math.}, vol.~135, pp.~284--305, 1909.

\bibitem{TzanakisWeger}
N.~Tzanakis and B.~M.~M. de~Weger, ``On the practical solution of the {T}hue
  equation,'' {\em J. Number Theory}, vol.~31, no.~2, pp.~99--132, 1989.

\bibitem{Gross}
B.~H. Gross, ``An elliptic curve test for {M}ersenne primes,'' {\em J. Number
  Theory}, vol.~110, no.~1, pp.~114--119, 2005.

\bibitem{DenommeSavin}
R.~Denomme and G.~Savin, ``Elliptic curve primality tests for {F}ermat and
  related primes,'' {\em J. Number Theory}, vol.~128, no.~8, pp.~2398--2412,
  2008.

\bibitem{Tsumura}
Y.~Tsumura, ``Primality tests for {$2^p\pm 2^{(p+1)/2}+1$} using elliptic
  curves,'' {\em Proc. Amer. Math. Soc.}, vol.~139, no.~8, pp.~2697--2703,
  2011.

\bibitem{Silverbergetal}
A.~Abatzoglou, A.~Silverberg, A.~V. Sutherland, and A.~Wong, ``A framework for
  deterministic primality proving using elliptic curves with complex
  multiplication,'' {\em Math. Comp.}, vol.~85, no.~299, pp.~1461--1483, 2016.

\end{thebibliography}

\end{document}